\documentclass[12pt,reqno]{amsart}
\usepackage{amssymb}
\usepackage{color}
\usepackage{amsthm}
\usepackage[pdftex]{graphicx}
\usepackage[margin=1in]{geometry}  

\newcommand{\I}{\mathcal{I}}
\newcommand{\h}{\mathcal{H}}
\newcommand{\ebarl}{\mathcal{\bar{L}}}
\newcommand{\el}{\mathcal{{L}}}

\newcommand{\e}{\mathcal{E}}

\newcommand{\gafa}{\gamma_{\fa}}
\newcommand{\ga}{\gamma}

\newcommand{\bfJ}{\boldsymbol{J}}
\newcommand{\bfOm}{\boldsymbol{\Omega}}
\newcommand{\la}{\lambda}
\newcommand{\ka}{\kappa}

\newcommand{\fa}{\alpha}
\newcommand{\grad}{\nabla}
\newcommand{\bfb}{\boldsymbol{b}}
\newcommand{\bfu}{\boldsymbol{u}}
\newcommand{\bfU}{\boldsymbol{U}}

\newcommand{\bfx}{\boldsymbol{x}}

\numberwithin{equation}{section}

\newtheorem{theorem}{Theorem}[section]
\newtheorem{proposition}[theorem]{Proposition}
\newtheorem{lemma}[theorem]{Lemma}

\newtheorem{remark}[theorem]{Remark}
\newtheorem{definition}{Definition}[section]

\newcommand{\om}{\omega}
\newcommand{\R}{\mathbb{R}}

\begin{document}

\title[On a Generalized System with Applications to Ideal MHD]{On a Generalized System with Applications to Ideal Magnetohydrodynamics}

\author{Alejandro Sarria}
\address{Department of Mathematics, University of North Georgia, Dahlonega, GA 30597}
\email{alejandro.sarria@ung.edu}




\begin{abstract}
Finite-time blowup of solutions $(u(x,t), b(x,t))$ to a generalized system of equations with applications to ideal Magnetohydrodynamics (MHD) and one-dimensional fluid convection and stretching, among other areas, is investigated. The system is parameter-dependent, our spatial domain is the unit interval or the circle, and the initial data $(u_0(x),b_0(x))$ is assumed to be smooth. Among other results, we derive precise blowup criteria for specific values of the parameters by tracking the evolution of $u_x$ along Lagrangian trajectories that originate at a point $x_0$ at which $b_0(x)$ and $b_0'(x)$ vanish. We employ concavity arguments, energy estimates, and ODE comparison methods. We also show that for some values of the parameters, a non-vanishing $b_0'(x_0)$ suppresses finite-time blowup.      
\end{abstract}

\maketitle

\section{Introduction}
\label{intro}

\subsection{Background} We are concerned with finite-time blowup of solutions to the generalized system 
\begin{equation}
\label{gmhd}
\begin{split}
&u_{xt}+uu_{xx}=\lambda u_x^2-\lambda b_x^2+\ka bb_{xx}+I(t),
\\
&b_{t}+ub_{x}=\ka bu_{x},
\\
&u_t+uu_x=-p_x+bb_x,
\\
&I(t)=(\la+\ka)\|b_x(\cdot,t)\|_{L^2([0,1])}^2-(\la+1)\|u_x(\cdot,t)\|_{L^2([0,1])}^2
\end{split}
\end{equation}
for $0\leq x\leq 1$, smooth initial data $(u(x,0),b(x,0))=(u_0(x),b_0(x))$, parameters $\la, \ka\in\R$, and either the Dirichlet boundary condition
\begin{align}
\label{dbc}
u(0,t)=u(1,t)=0,\quad b(0,t)=b(1,t)=0,
\end{align}
or the periodic boundary condition
\begin{equation}
\begin{split}
\label{pbc}
&u(0,t)=u(1,t),\,\, u_x(0,t)=u_x(1,t),
\\
&b(0,t)=b(1,t),\,\, b_x(0,t)=b_x(1,t).
\end{split}
\end{equation}
Our motivation for studying \eqref{gmhd} is threefold. First, for $\lambda=\frac{1}{n-1},\, n\in\mathbb{Z}^+,\, n\geq 2$, and $\ka=1$, \eqref{gmhd} describes infinite energy solutions of the $n-$dimensional incompressible Magnetohydrodynamics (MHD) equations with zero dissipation
\begin{equation}
\label{inviscidMHD}
\begin{split}
&\bfu_t+\bfu\cdot\nabla\bfu=-\nabla p+\bfb\cdot\grad\bfb
\\
&\bfb_t+\bfu\cdot\grad\bfb=\bfb\cdot\grad\bfu
\\
&\grad\cdot\bfu=\grad\cdot\bfb=0, 
\end{split}
\end{equation}
obtained under a stagnation-point similitude velocity field and magnetic field   
\begin{equation}
\label{spf}
\begin{split}
\bfu(x,\bfx',t)=\left(u(x,t),-\frac{\bfx'}{n-1}u_x(x,t)\right),\quad \bfb(x,\bfx',t)=\left(b(x,t),-\frac{\bfx'}{n-1}b_x(x,t)\right)
\end{split}
\end{equation}
for $\bfx'=(x_2,x_3,\cdots, x_n)$. The term ``stagnation-point similitude'' arises from the observation that velocity fields of the form \eqref{spf}-i) emerge from the modeling of flow near a rear stagnation point (\cite{stuart}, \cite{ohkitani1}, \cite{gibbon}).

The ideal MHD equations \eqref{inviscidMHD} couples the incompressible Euler equations with an external magnetic field $\bfb$ for $p=P+\frac{1}{2}|\bfb|^2$ the scalar magnetic pressure, $P$ the hydrodynamic pressure and $\bfu$ the fluid velocity. Although the global regularity problem for the fully dissipative incompressible 2D MHD equations
\begin{equation}
\label{MHD}
\begin{split}
&\bfu_t+\bfu\cdot\nabla\bfu=-\nabla p+\bfb\cdot\grad\bfb+\nu\Delta\bfu,
\\
&\bfb_t+\bfu\cdot\grad\bfb=\bfb\cdot\grad\bfu+\mu\Delta\bfb,
\\
&\grad\cdot\bfu=\grad\cdot\bfb=0,
\end{split}
\end{equation}
has long been settled (see, e.g., \cite{Sermange}), it is not known if smooth solutions of the fully dissipative 3D MHD equations, or the 2D and 3D ideal MHD systems \eqref{inviscidMHD}, blowup in finite time. There is a large body of literature on the 2D and 3D MHD equations with full, partial, or zero dissipation. The reader may refer to the survey by Wu (\cite{Wu}) for a summary of relatively recent developments and additional references. 

Finite-time blowup of infinite energy solutions of the MHD equations, and related systems such as the incompressible Euler, inviscid Boussinesq, and incompressible Navier-Stokes equations, has been a particularly active area of research for the insights that the study of this particular class of solutions may provide into the global regularity problems involving their finite energy counterparts. So-called ``stretched'' infinite energy solutions  
\begin{equation}
\label{2andahalf}
\begin{split}
&\bfu(x,y,z,t)=(u_1(x,y,t),u_2(x,y,t),z\ga_1(x,y,t)+V(x,y,t))
\\
&\bfb(x,y,z,t)=(b_1(x,y,t),b_2(x,y,t),z\ga_2(x,y,t)+W(x,y,t))
\end{split}
\end{equation} 
of the 3D MHD equations have been investigated by several authors. In particular, solutions were shown to exist globally in time for $\ga_1=\ga_2\equiv0$ (\cite{chae0}). Moreover, and of particular interest in this paper, Gibbon and Ohkitani (\cite{gibbon3}) studied periodic solutions of the 3D ideal MHD equations \eqref{inviscidMHD} of the form \eqref{2andahalf} on an infinite tubular domain. Their numerical experiments suggest finite-time blowup in the velocity gradient and a simultaneous, yet late-stage and hard to detect, blowup in the magnetic field that was contingent upon the initial magnetic field $\bfb_0$ not vanishing at the spatial location(s) where blowup in the velocity gradient was found. If such vanishing was to occur, then there would be velocity gradient blowup but no singularity in the magnetic field. Additionally, Yan (\cite{Yan}) proved finite-time blowup of infinite energy, self-similar solutions in $\R^3$ for the 3D incompressible MHD equations \eqref{MHD} with smooth and non-smooth infinite energy initial data. In particular, the blowup was found to occur in the velocity gradient but not the magnetic field, and the blowup result extends to the 3D ideal MHD equations and the 3D Navier-Stokes equations. In contrast, global-in-time existence of stretched, infinite energy solutions of the 2D ideal MHD equations in $\R^2$ was established in \cite{Minkyu} for positive Elsasser initial data. With these results in mind, and noticing that the vanishing of $b(x,t)$ and $b_x(x,t)$ in \eqref{spf}-ii) at a point $x=x_0\in[0,1]$ implies the vanishing of the entire field $\bfb(x,\bfx',t)$ at $x=x_0$, one of the recurring themes in this work will be the role that a zero of $b(x,t)$ and/or $b_x(x,t)$ may play in the formation (or suppression) of a finite-time singularity in a solution of the generalized system \eqref{gmhd}. For additional results on the global regularity of infinite energy solutions of the incompressible MHD, Euler, inviscid Boussinesq, and Navier-Stokes equations, we direct the reader to \cite{gibbon1, gibbon4, constantin2, childress, Okamoto1, Sarria0, Sarria01, sarria1, sarria2}. 

Our second motivation for studying \eqref{gmhd} has its origin in the work \cite{ohkitani0} by Ohkitani and Okamoto where the authors argued that fluid convection can have a regularizing effect in the global existence problem for hydrodynamically relevant evolution equations (\cite{okamoto0, hou22}). In particular, if we allow for arbitrary parameter values $(\la,\ka)\in\R^2$, then \eqref{gmhd} may be used as a tool to study the interplay between one-dimensional fluid convection and stretching as follows. Noting that the induction equation \eqref{inviscidMHD}ii) can be written as $\bfb_t=\nabla\times(\bfu\times\bfb)$, we take the curl of \eqref{inviscidMHD}i)-ii) and use the standard vector identities
\begin{equation*}
\begin{split}
& \nabla\times(\nabla\times\bf{F})=\nabla(\nabla\cdot\bf{F})-\Delta\bf{F}
\\
& \nabla\cdot(\bf{F}\times\bf{G})=\bf{G}\cdot(\nabla\times\bf{F})-\bf{F}\cdot(\nabla\times\bf{G})
\\
&\nabla(\bf{F}\cdot\bf{G})=\bf{F}\cdot\nabla\bf{G}+\bf{G}\cdot\nabla\bf{F}+\bf{F}\times(\nabla\times\bf{G})+\bf{G}\times(\nabla\times\bf{F}),
\end{split}
\end{equation*}
along with \eqref{inviscidMHD}iii), to obtain the vorticity and current density formulation of the ideal 3D MHD equations, 
\begin{equation}
\label{3dmhd}
\begin{split}
&\bfOm_{t}+\underbrace{\bfu\cdot\nabla\bfOm}_{convection}=\underbrace{\bfOm\cdot\nabla\bfu}_{stretching}\,-\bfJ\cdot\nabla \bfb+\bfb\cdot\nabla\bfJ
\\
&\bfJ_{t}+\underbrace{\bfu\cdot\nabla \bfJ}_{convection}=\bfb\cdot\nabla\bfOm+\bfOm\cdot\nabla\bfb-\underbrace{\bfJ\cdot\nabla\bfu}_{stretching}+\,\bf{H}
\\
&\bfOm=\nabla\times\bfu,\,\, \bfJ=\nabla\times\bfb,
\end{split}
\end{equation} 
where $\bfOm$ is the vorticity, $\bfJ$ the current density, and $\bf{H}$ stands for the higher-order terms
$$\bf{H}=\bfb\times(\nabla\times\bfOm)+(\bfOm\times\bfJ)-\bfu\times(\nabla\times\bfJ)-(\bfJ\times\bfOm)-\Delta(\bfu\times\bfb).$$
Now, in 3D, \eqref{spf} gives   
\begin{align*}
&\nabla\times \bfu=\nabla\times\left(u,-\frac{y}{2}u_x,-\frac{z}{2}u_x\right)=(0,z/2,-y/2)u_{xx}(x,t),
\\
&\nabla\times \bfb=\nabla\times\left(b,-\frac{y}{2}b_x,-\frac{z}{2}b_x\right)=(0,z/2,-y/2)b_{xx}(x,t),
\end{align*}
from which it follows that the evolution of the fluid vorticity and current density associated with solutions of the form \eqref{spf} of the 3D ideal MHD equations \eqref{inviscidMHD} (and actually the 2D case as well) is dictated by the scalar functions $\omega=u_{xx}$ and $j=b_{xx}$. If we now differentiate \eqref{gmhd}-i)-ii) with respect to $x$, we obtain
\begin{equation}
\label{vorticityform}
\begin{split}
&\om_{t}+\underbrace{u\om_x}_{convection}=(2\la-1)\underbrace{\om u_x}_{stretching}-(2\la-\ka)jb_x+\ka bj_x
\\
&j_{t}+\underbrace{uj_x}_{convection}=\ka b\om_x+(2\ka-1)\om b_x-(2-\ka)\underbrace{ju_x}_{stretching}
\\
&\om=u_{xx},\,\, j=b_{xx},
\end{split}
\end{equation}
which may serve as a one-dimensional analogue of \eqref{3dmhd}. Thus, by varying the parameters $\la, \ka\in\R$, we may better understand the competing effects between one-dimensional fluid convection, stretching, and coupling in the formation of singularities (or their suppression). More particularly, the quadratic terms in \eqref{vorticityform} represent the competition in fluid convection between nonlinear steepening and amplification due to $(2\lambda-1)$-dimensional stretching and $2\la-\kappa$ and $\ka$-dimensional coupling (\cite{holm}), with the parameters $\lambda$ and $\kappa$ measuring the ratio of stretching to convection and the impact of the coupling between $u$ and $b$, respectively (\cite{wunsch0, wunsch1}).

Our third motivation for studying \eqref{gmhd} is that for particular values of the parameters $\la, \ka\in\R$, \eqref{gmhd} interpolates between several well-known models from fluid dynamics and mathematical physics in general. Indeed, in addition to describing solutions of the $n-$dimensional ideal MHD equations of the form \eqref{spf}, we also have that for $b(x,0)=b_0(x)\equiv0$, \eqref{gmhd} reduces to the generalized inviscid Proudman-Johnson equation (gPJ) \cite{Proudman1, Okamoto1, Wunsch11, Sarria0, Sarria01} comprising, among others, the Burgers' equation of gas dynamics for $\lambda=-1$, stagnation point-form solutions \eqref{spf}-i) of the $n-$dimensional incompressible Euler equations when $\lambda=\frac{1}{n-1}$ (\cite{childress, Saxton1}), and the Hunter-Saxton equation describing the orientation of waves in massive director field nematic liquid crystals for $\lambda=-\frac{1}{2}$ (\cite{Hunter1}). The gPJ equation also admits a geometric interpretation as a geodesic equation on the group of orientation-preserving diffeomorphisms of the circle (modulo rotations) (\cite{bauer}).

Moreover, for $(\la,\ka)=(-1/2,-1)$ and $b(x,t)=\rho(x,t)$ satisfying $$\rho(x,t)=\int_0^x(x-y)\rho(y,t)\,dy+\frac{1}{2}f(t)x^2+g(t)x+h(t)$$
for arbitrary functions $f, g,$ and $h$, the system \eqref{gmhd}-i)-ii) (with \eqref{gmhd}-i) differentiated once in space) becomes the Hunter-saxton (HS) system
\begin{equation}
\label{HS}
\begin{split}
&u_{xxt}+uu_{xxx}+2u_xu_{xx}+\rho\rho_x=0, 
\\
&\rho_t+u\rho_x=-\rho u_x.
\end{split}
\end{equation}
The HS system has applications in shallow-water theory and other areas. As a shallow-water model, $u$ and $\rho$ represent the horizontal velocity and density, respectively. The HS system arises as the ``short-wave'' limit of the Camassa-Holm (CH) system (\cite{Aconstantin0, Escher0}), which is in turn derived from the Green-Naghdi equations (\cite{Johnson1}), widely used in coastal oceanography to approximate the free-surface Euler equations. The HS system \eqref{HS} is also a particular case of the Gurevich-Zybin system describing the formation of large scale structure in the universe (\cite{Pavlov1}).

Lastly, for $\la=1$ and $\ka=3/2$, the trace of \eqref{gmhd}-i)-ii) with the Dirichlet boundary condition \eqref{dbc} describes stagnation-point form solutions
\begin{equation}
\label{bspf}
\begin{split}
\bfu(x,y,t)=(u(x,t),-yu_x(x,t)),\quad\theta(x,y,t)=yb(x,t)
\end{split}
\end{equation}
of the inviscid 2D Boussinesq equations
\begin{equation}
\label{boussinesq}
\begin{split}
&\bfu_t+\bfu\cdot\nabla\bfu=-\nabla p+\theta\,e_{2},
\\
&\theta_t+\bfu\cdot\nabla\theta=0,
\\
&\nabla\cdot\bfu=0,
\end{split}
\end{equation}
along the boundary $x\in\{0,1\}$ with Dirichlet boundary conditions. The Boussinesq equations model large scale atmospheric and oceanic flows responsible for cold fronts and the jet stream (see, e.g., \cite{gill}, \cite{majda}). Mathematically, the 2D Boussinesq equations serve as a lower-dimensional model of the 3D hydrodynamics equations and retain some key features of the 3D Euler equations, such as vortex stretching. In addition, it is well-known that (away from the axis of symmetry) the inviscid 2D Boussinesq equations are closely related to the Euler equations for 3D axisymmetric swirling flows (\cite{majdabertozzi}). 

\begin{remark}
Local well-posedness of \eqref{gmhd}-\eqref{dbc} (or \eqref{gmhd} with \eqref{pbc}) can be established via standard energy estimates by following an argument similar to that done in \cite{sarria2}. In particular, it can be shown that if $u_0\in H^2([0,1])$, $u_0'\in L^\infty([0,1])$, and $b_0\in H^2([0,1])$, then there exists a time $T=T(\|u_0\|_{H^2}, \|u_0'\|_{L^\infty}, \|b_0\|_{H^2})>0$
such that \eqref{gmhd} has a unique solution on $[0,T]$ satisfying $u\in C([0,T]; H^2)$,
$u_x\in C([0,T]; L^\infty)$ and $b\in C([0,T]; H^2)$.
Moreover, similar to the continuation criterion established in \cite{sarria2} and in the same spirit as the continuation criterion for ideal MHD solutions (see, e.g., \cite{Wu}), we have that if
\begin{equation} \label{Cricon}
\int_0^{T^*}\left( \|u_x(t)\|_{L^\infty}+\|b_x(t)\|_{L^\infty}\right)dt <+\infty,
\end{equation}
then the local solution can be extended to $[0,T^*]$. It is also worth noting that Kato's local existence theory (\cite{Kato}) can also be used to establish local well-posedness of \eqref{gmhd} in a manner similar as it was used in \cite{wunsch1, wunsch0} to establish local well-posedness of solutions to the generalized Hunter-Saxton system. Lastly, the reader may also refer to \cite{Kim} and \cite{Minkyu} for local well-posedness results for infinite-energy solutions of the type considered here in other Sobolev and Besov spaces.
\end{remark}

\subsection{Outline and Summary of Main Results}
The outline for the remainder of the paper is as follows. In section \ref{prelim}, we establish some preliminary results needed in later sections. In section \ref{sec:concavity}, we prove finite-time blowup of solutions of \eqref{gmhd} for $-1\leq \la<0$ and $\ka\leq-\la$ using a concavity argument on the Jacobian of the Lagrangian trajectories. In section \ref{sec:energy}, we further elaborate on the finite-time singularity for $(\la,\ka)\in\{-1/2\}\times(-\infty,1/2]$ using an energy conservation law. Even though these parameter values fall under those studied in section \ref{sec:concavity}, we include them here as we uncover additional regularity properties that the blowup solutions in section \ref{sec:concavity} do not necessarily possess. Then, in section \ref{sec:pje}, we show that for $\la\in\R\backslash\{0\}$ and $\ka=-\la$, solutions of \eqref{gmhd} restricted to a particular family of Lagrangian paths behave just like solutions of an Euler-related equation known to undergo finite-time blowup. In section \ref{sec:comparison}, we use a standard Sturm comparison argument to derive a blowup criterion for solutions of the 2D ideal MHD equations of the form \eqref{spf} in terms of the associated 2D Euler hydrodynamic pressure and 2D MHD total pressure. All blowup results in sections \ref{sec:concavity}-\ref{sec:comparison} require the vanishing of $b(x,0)=b_0(x)$ and $b_x(x,0)=b_0'(x)$ at a point $x_0\in[0,1]$. As a result of the ansatz \eqref{spf}, this implies singularity formation at a point at which the vector-valued function $\bfb(x,\bfx',t)$ is identically zero. In section \ref{sec:suppressing}, we show that for particular values of $\la$ and $\ka$, removing the vanishing condition from $b_0'(x_0)$ (and, thus, from $\bfb(x_0,\bfx',t)$) turns a previously singular solution into a global one. Lastly, in section \ref{sec:zero}, we prove global-in-time existence in the simple case $(\la,\ka)=(0,0)$.

The table below summarizes most of the regularity results in this paper. Assuming smooth initial data, the table describes parameter values for which $u_x$ and $b_{xx}$ leave $L^{\infty}([0,1])$ in finite time, with a couple of rows referencing the behavior of $b_x$ as well. An empty cell means we have no definite information about the evolution of any $L^p$ norms of the function in question, although in many cases we do know about their pointwise behavior at certain locations in $[0,1]$. See the Theorem referenced in the table for details on such behavior as well as additional assumptions placed on the initial data.    
\begin{table}[h]
\centering
\begin{tabular}{|c|c|c|c|c|}
\hline
$(\la,\ka)$  & $u_x$  & $b_{xx}$  & $b_x$   & \text{See Theorem}  \\ \hline
$[-1,0)\times(-\infty,-\la]$   & $\notin L^{\infty}$    & $\notin L^{\infty}$    &   & \ref{concave}   \\ \hline

$\{-\frac{1}{2}\}\times\left(-\infty,\frac{1}{2}\right]$   & $\notin L^{\infty}, \in L^2$    & $\notin L^{\infty}$    & $\in L^2$   & \ref{energyconserved}   \\ \hline

$\left(\frac{1}{2},1\right]\times\{-\la\}$   & $\notin L^{\infty}$    & $\notin L^{\infty}$    &     & \ref{lambdapos}    \\ \hline

$(1,\infty)\times\{-\la\}$   & $\notin L^{\infty}$    & $\in L^{\infty}$    &     & \ref{lambdapos}   \\ \hline

$(-2,0)\times\{-\la\}$   & $\notin L^{\infty}$   & $\notin L^{\infty}$    &     & \ref{lambdaneg}    \\ \hline

$(-\infty,-2]\times\{-\la\}$   & $\notin L^{\infty}$   & $\in L^{\infty}$    &     & \ref{lambdaneg}    \\ \hline

$(0,0)$   & $\in L^{\infty}$    & $\in L^{\infty}$    & $\in L^{\infty}$   & \ref{global2}    \\ \hline

$\left(0,\frac{1}{2}\right]\times\{-\la\}$   &     &     &    &    \\ \hline

$\R\times(-\la,\infty)$   &     &     &     &    \\ \hline

$\left[(-\infty,-1)\cup [0,\infty)\right]\times(-\infty,-\la)$   &     &    &     &    \\ \hline

\end{tabular}
\end{table}

\newpage

\section{Preliminary Results}
\label{prelim}

Define the Lagrangian trajectories $\ga$ through the IVP
\begin{align}
\label{lagrangian}
\ga_t(\fa,t)=u(\ga(\fa,t),t),\qquad \ga(\fa,0)=\fa
\end{align}
for $\fa\in[0,1]$. Differentiating \eqref{lagrangian} gives
\begin{align}
\label{gafa}
\ga_{\fa t}=u_x(\ga(\fa,t),t)\cdot\gafa,
\end{align}
whose solution, the Jacobian $\gafa$, is given by
\begin{align}
\label{jacobian}
\gafa(\fa,t)=e^{\int_0^tu_x(\ga(\fa,s),s)\,ds}.
\end{align}
Next, recall the notion of the order of a zero of a function.
\begin{definition}
We say a non-trivial, $m-$times differentiable function $f(x)$ has a zero of order $m\in\mathbb{Z}^+$ at $x=x_0$, if $$f(x_0)=f'(x_0)=\cdots=f^{(m-1)}(x_0)=0$$ 
but $f^{(m)}(x_0)\neq0$.
\end{definition}
Lemma \ref{lemma:order} below states that the order of a zero of the initial condition $b(x,0)=b_0(x)$ is preserved by $b(x,t)$ along Lagrangian trajectories for as long as a solution exists. This property will be central to most blowup results in this paper. 
\begin{lemma}
\label{lemma:order}
Consider \eqref{gmhd} with the Dirichlet boundary condition \eqref{dbc} or the periodic boundary condition \eqref{pbc}. Suppose the initial data $(u(x,0),b(x,0))=(u_0(x),b_0(x))$ is smooth and $b_0(\fa)$ has a zero of order $m\in\mathbb{Z}^+$ at $\fa=\fa_0$. Then, the order $m$ is preserved by $b(x,t)$ along $\ga(\fa_0,t)$ for as long as a solution exists, i.e., for all $0\leq n\leq m-1$,
\begin{align}
\label{order}
\frac{\partial^{n}b}{\partial x^{n}}(x,t)\circ\ga(\fa_0,t)\equiv0, \qquad\text{while}\qquad \frac{\partial^mb}{\partial x^m}(x,t)\circ\ga(\fa_0,t)\neq 0.
\end{align}
Additionally, for all $\ka\in\R$, the $m_\text{th}$ spatial partial of $b(x,t)$ along $\ga(\fa_0,t)$ satisfies  
\begin{align}
\label{orderpde}
\frac{d}{dt}\left[\left(\frac{\partial^mb}{\partial x^m}(x,t)\right)\circ\ga(\fa_0,t)\right]=(\ka-m)\left(u_x(x,t)\frac{\partial^mb}{\partial x^m}(x,t)\right)\circ\ga(\fa_0,t)
\end{align}
with solution 
\begin{align}
\label{jacorder}
\frac{\partial^mb}{\partial x^m}(x,t)\circ\ga(\fa_0,t)=b_0^{(m)}(\fa_0)\cdot\gafa^{\ka-m}(\fa_0,t).
\end{align}
 
\end{lemma}

\begin{proof}
The order conservation \eqref{order} follows directly from equation \eqref{gmhd}-ii) and definition \eqref{lagrangian} of the Lagrangian trajectories $\ga$ as follows. The evolution of $b$, $b_x$ and $b_{xx}$ along $\ga$ are given by
\begin{align*}
\frac{d}{dt}\left(b(\ga(\fa,t),t)\right)&=(\ka bu_x)\circ\ga(\fa,t),
\\
\frac{d}{dt}\left(b_x(\ga(\fa,t),t)\right)&=\left((\ka-1) u_xb_x+\ka bu_{xx}\right)\circ\ga(\fa,t),
\\
\frac{d}{dt}\left(b_{xx}(\ga(\fa,t),t)\right)&=\left((\ka-2) u_xb_{xx}+(2\ka-1)b_x u_{xx}+\ka bu_{xxx}\right)\circ\ga(\fa,t).
\end{align*}
Solving the first equation gives
\begin{align}
\label{balong}
b(\ga(\fa,t),t)=b_0(\fa)\cdot\gafa^{\ka}(\fa,t),
\end{align}
where the Jacobian, $\gafa$, which is positive and finite for as long as $\ga$ remains a diffeomorphism of the unit interval, is given by \eqref{jacobian}. From \eqref{balong}, it follows that if $b_0(\fa_0)=0$, then $b(\ga(\fa_0,t),t)\equiv0$ for as long as $\gafa(\fa_0,t)>0$ is defined. Assume $b_0(\fa_0)=0$. Then setting $\fa=\fa_0$ into the second equation and using the vanishing of $b\circ\ga(\fa_0,t)$ gives
$$\frac{d}{dt}\left(b_x(\ga(\fa_0,t),t)\right)=(\ka-1)\left(u_xb_x\right)\circ\ga(\fa_0,t)$$
with solution 
$$b_x(\ga(\fa_0,t),t)=b_0'(\fa_0)\cdot\gafa^{\ka-1}(\fa_0,t).$$
Thus, if in addition to $b_0(\fa_0)=0$ we now assume that $b_0'(\fa_0)=0$, then $b(\ga(\fa_0,t),t)\equiv0$ and $b_x(\ga(\fa_0,t),t)\equiv0$ for as long as a solution exists. Using the same procedure on the third equation shows that if $b_0(\fa_0)=b_0'(\fa_0)=b_0''(\fa_0)=0$, then $b(\ga(\fa_0,t),t)=b_x(\ga(\fa_0,t),t)=b_{xx}(\ga(\fa_0,t),t)\equiv0$. Subsequent higher-order spatial derivatives of $b$ along $\ga$ can be shown to follow this pattern in much the same way. In fact, a simple induction argument, which we omit for the sake of brevity, suffices to establish \eqref{order} as well as equations $(\ref{orderpde})$ and \eqref{jacorder}. This finishes the proof of the Lemma.
\end{proof}
Our next Lemma will be used in section \ref{sec:energy} to prove finite-time blowup for $\lambda=-1/2$ and $\ka\leq 1/2$. To simplify the notation, we write
$$\|f\|_p=\|f\|_{L^p([0,1])},\qquad 1\leq p\leq\infty.$$
\begin{lemma}
\label{lemma:energyconserved}
Consider the system \eqref{gmhd} with the Dirichlet boundary condition \eqref{dbc} or the periodic boundary condition \eqref{pbc} and smooth initial data $(u_0(x),b_0(x))$. Set 
\begin{align}
\label{energy}
\e(t)=\|u_x\|_2^2+\|b_x\|_2^2
\end{align}
and suppose $(\la,\ka)\in\{-1/2\}\times\R$. Then 
\begin{align}
\label{conserved}
\e(t)=\e(0)
\end{align}
for as long as solutions exist. Moreover, for $I(t)$ the nonlocal term in \eqref{gmhd}-iv) and $\e_0=\e(0)$,  
\begin{align}
\label{k>=0}
-\frac{1}{2}\e_0\leq I(t)\leq\left(\ka-\frac{1}{2}\right)\e_0
\end{align}
for $\ka\geq 0$, while
\begin{align}
\label{k<0}
\left(\ka-\frac{1}{2}\right)\e_0\leq I(t)\leq-\frac{1}{2}\e_0
\end{align}
for $\ka<0$.
\end{lemma}
\begin{proof}
Multiplying \eqref{gmhd}-i) by $u_x$, integrating, and then using either set of boundary conditions, \eqref{dbc} or \eqref{pbc}, yields
$$\frac{d}{dt}\int_0^1u_x^2\,dx=(2\la+1)\int_0^1u_x^3\,dx-2(\la+\ka)\int_0^1u_xb_x^2\,dx-2\ka\int_0^1bb_xu_{xx}\,dx.$$
Similarly, taking the $x-$derivative of \eqref{gmhd}-ii), multiplying by $b_x$ and integrating gives 
$$\frac{d}{dt}\int_0^1b_x^2\,dx=(2\ka-1)\int_0^1u_xb_x^2\,dx+2\ka\int_0^1bb_xu_{xx}\,dx.$$
Adding the two equations, we obtain 
$$\e'(t)=(2\la+1)\left(\int_0^1u_x^3\,dx-\int_0^1u_xb_x^2\,dx\right).$$
Setting $\la=-1/2$ and then integrating yields \eqref{conserved}. 

For the second part of the Lemma, set $\e(0)=\e_0$ and note that $\e(t)=\e_0$ implies that the nonlocal term $I(t)$ in \eqref{gmhd}-iv) satisfies
\begin{equation}
\label{bounds}
\begin{split}
I(t)&=\left(\ka-\frac{1}{2}\right)\|b_x\|_2^2-\frac{1}{2}\|u_x\|_2^2
\\
&=\left(\ka-\frac{1}{2}\right)\left(\e_0-\|u_x\|_2^2\right)-\frac{1}{2}\|u_x\|_2^2
\\
&=\left(\ka-\frac{1}{2}\right)\e_0-\ka\|u_x\|_2^2\,.
\end{split}
\end{equation}
Since \eqref{conserved} implies that $0\leq \|u_x\|_2^2\leq \e_0$, \eqref{k>=0} and \eqref{k<0} follow from \eqref{bounds}. This finishes the proof of the Lemma.
\end{proof}

\section{Finite-time blowup for $-1\leq \la<0$ and $\ka\leq-\la$}
\label{sec:concavity}
In this section, we show that if $b_0(\fa)$ has a zero of order $m\geq2$ at some $\fa=\fa_0\in[0,1]$, then $\ga^{|\la|}_{\fa}(\fa_0,t)$ is concave for $\la\in[-1,0)$ and $\ka\leq-\la$. This will imply a finite-time singularity in the form of a vanishing Jacobian $\gafa(\fa_0,t)$. 

\begin{theorem}
\label{concave}
Consider the system \eqref{gmhd} with the Dirichlet boundary condition \eqref{dbc} or the periodic boundary condition \eqref{pbc} and smooth initial data $(u_0(\fa),b_0(\fa))$. Suppose $b_0(\fa)$ has a zero of order $m\geq 2$ at $\fa=\fa_0$, i.e., at least $b_0(\fa)$ and $b_0'(\fa)$ vanish at $\fa_0$, and $u_0'(\fa)$ attains its absolute minimum $m_0<0$ at $\fa_0$. Then, for $\la\in[-1,0)$, $\ka\leq-\la$ and $T=\frac{1}{\la m_0}$, 
\begin{align}
\label{concavityblow}
\lim_{t\nearrow T}\int_0^t{u_x(\ga(\fa_0,s),s)\,ds}=-\infty
\end{align}
and
\begin{align}
\label{concavityblow2}
\lim_{t\nearrow T}\bigg|\frac{\partial^mb}{\partial x^m}(x,t)\circ\ga(\fa_0,t)\bigg|=\infty.
\end{align}
\end{theorem}
\begin{proof}
Differentiating \eqref{gafa} with respect to $t$ and then using \eqref{gmhd}-i) and \eqref{gafa} gives
\begin{equation}
\label{concavity-1}
\begin{split}
\ga_{\fa tt}&=\left(u_x^2\circ\ga\right)\cdot\gafa+\left(\left(\la u_x^2-\la b_x^2+\ka bb_{xx}\right)\circ\ga+I(t)\right)\cdot\gafa
\\
&=
(1+\la)\frac{\ga^2_{\fa t}}{\gafa}+\left(\left(\ka bb_{xx}-\la b_x^2\right)\circ\ga+I(t)\right)\cdot\gafa,
\end{split}
\end{equation}
which can be rearranged as
\begin{align}
\label{concavity0}
\frac{\gamma_{\fa}\cdot\gamma_{\fa tt}-(1+\la)\gamma^2_{\fa t}}{\gamma^2_{\fa}}=\left(\ka bb_{xx}-\la b_x^2\right)\circ \gamma+I(t)
\end{align}
or, equivalently, 
\begin{align}
\label{concavity1}
-\frac{1}{\la}\gamma_{\fa}^{\la}\cdot\partial_t^2\left(\gamma_{\fa}^{-\la}\right)=(\ka bb_{xx}-\la b_x^2)\circ\gamma+I(t)
\end{align}
for $\la\neq 0$ and $I(t)=(\la+\ka)\|b_x\|_2^2-(1+\la)\|u_x\|_2^2$. Let $\fa_0\in[0,1]$ be such that $b_0(\fa)$ has a zero of order $m\geq 2$ at $\fa_0$. By Lemma \ref{lemma:order}, this implies that $b(\ga(\fa_0,t),t)\equiv0$ and $b_x(\ga(\fa_0,t),t)\equiv0$ for as long as solutions exist. Consequently, setting $\fa=\fa_0$ in \eqref{concavity1} yields
\begin{align}
\label{last}
-\frac{1}{\la}\gamma_{\fa}^{\la}(t)\cdot\frac{d^2}{dt^2}\left(\gamma_{\fa}^{-\la}(t)\right)=I(t)
\end{align}

for $\gamma_{\fa}(t)=\gamma_{\fa}(\fa_0,t)$. Suppose $\ka\leq -\la$ for $-1\leq\la<0 $. Then  
$$(\la+\ka)\|b_x\|_2^2-(1+\la)\|u_x\|_2^2\leq 0,$$
so that 
$$-\frac{1}{\la}\gamma_{\fa}^{\la}(t)\cdot\frac{d^2}{dt^2}\left(\gamma_{\fa}^{-\la}(t)\right)\leq 0.$$
Since $\gamma_{\fa}(t)>0$ for as long as solutions exist and $-1\leq \la<0$, we must have that 
\begin{align}
\label{ineq1}
\frac{d^2}{dt^2}\left(\gamma_{\fa}^{|\la|}(t)\right)\leq 0,
\end{align}
which implies that $\gamma_{\fa}^{|\la|}(t)$ is concave. In particular, this implies that  the graph of $\gamma_{\fa}^{|\la|}(t)$ must lie below its tangent line at $t=0$. Suppose $u_0'(\fa)$ attains its absolute minimum $m_0<0$ at $\fa_0$. Then, using $$\ga_{\fa t}(\fa_0,t)=\gafa(\fa_0,t)\cdot u_x(\ga(\fa_0,t),t),\quad \gamma_{\fa t}(\fa_0,0)=u_0'(\fa_0),\quad \gafa(\fa_0,0)=1,\quad u_0'(\fa_0)=m_0,$$
along with concavity of $\gamma_{\fa}^{|\la|}(t)$, we obtain
\begin{align}
\label{ineq2}
0<\gamma_{\fa}^{|\la|}(t)\leq 1-\la m_0t.
\end{align}
Taking the limit as $t\nearrow T=\frac{1}{\la m_0}$ in \eqref{ineq2} and recalling \eqref{jacobian} yields \eqref{concavityblow}.   

Lastly, since $-1\leq \la<0$, we have that $\ka\leq|\la|\leq 1$, and since $m\geq 2$, this implies that $\ka-m<0$. Using this and the vanishing of $\ga_{\fa}(\fa_0,t)$ as $t\nearrow T$ from \eqref{ineq2} on formula \eqref{jacorder} in Lemma \ref{lemma:order} yields
\begin{align}
\label{concavityblow3}
\lim_{t\nearrow T}\bigg|\frac{\partial^mb}{\partial x^m}(x,t)\circ\ga(\fa_0,t)\bigg|=\lim_{t\nearrow T}\frac{|b_0^{(m)}(\fa_0)|}{\ga_{\fa}(\fa_0,t)^{|m-\ka|}}=\infty,
\end{align}
where we also used $b_0^{(m)}(\fa_0)\neq 0$, a consequence of $b_0(\fa)$ having a zero of order $m\geq 2$ at $\fa_0$. This concludes the proof of the Theorem.
\end{proof}

Finite-time blowup of $\|u_x\|_{\infty}$ for the above parameter values is actually easier to obtain directly from equation \eqref{gmhd}-i). However, the blowup argument in Theorem \ref{concave} above highlights the strength of the singularity in the sense that the blowup is not time-integrable, something that, as we will see in Section \ref{sec:pje}, is not always the case. 

\begin{remark}
What if we have convexity? Suppose $\la+\ka\geq 0$ and $1+\la\leq 0$ in \eqref{last}. Then 
$$\frac{d^2}{dt^2}\left(\gamma_{\fa}^{|\la|}(t)\right)\geq 0,$$
and so $\gamma_{\fa}^{|\la|}(t)$ is now convex. Therefore, 
$$\gamma_{\fa}^{|\la|}(t)\geq 1+|\la|u_0'(\fa_0)t.$$
If $\fa_0$ is now chosen such that $u_0'(\fa_0)>0$, then $\gamma_{\fa}^{|\la|}(t)\geq1+|\la|u_0'(\fa_0)t>1$
for all $t\geq 0$ for which a solution is defined. This means that if there is finite-time blowup for $\la+\ka\geq 0$ and $1+\la\leq 0$ under the above assumptions on $u_0$ and $b_0$, such singularity cannot come in the form of a vanishing $\gamma_{\fa}(\fa_0,t)$, but as blowup of $u_x(\gamma(\fa_0,t),t)$ to $+\infty$; it can be shown that blowup to $-\infty$ in $u_x(\ga(\fa_0,t),t)$ cannot happen by working directly with \eqref{gmhd}-i), which gives $\partial_t(u_x\circ\ga(\fa_0,t))\geq \la u_x^2\circ\ga(\fa_0,t)$ for $\la\leq -1$. This implies, after solving, that if $u_0'(\fa_0)>0$ and there is blowup in $u_x$ along this particle trajectory, then it must be to $+\infty$ as $u_x$ starts off positive at $t=0$. Moreover, note that in this case we do not require $b_0'(\fa_0)$ to vanish since $\la<0$ implies $-\la b^2_x\circ\ga(\fa_0,t)\geq 0$. So we only need $b_0(\fa_0)=0$. 
\end{remark}

\section{Additional Regularity Results for $(\la,\ka)\in\{-1/2\}\times(-\infty,1/2]$}
\label{sec:energy}
The range of parameters to which the blowup result in this section applies falls under the parameter values covered in Theorem \ref{concave} of the previous section. Nevertheless, we include it here for two reasons. First, the result in this section uses the \textit{energy} conservation Lemma \ref{lemma:energyconserved} instead of a concavity argument. This implies that, even though we find the same finite-time blowup in the $L ^{\infty}([0,1])$ norm that we did in section \ref{sec:concavity}, the singularity here appears to be \textit{milder} than the blowup in Theorem \ref{concave} in the sense that solutions remain in other $L^p$ spaces for $1\leq p<\infty$. Second, the solution's energy being conserved potentially implies that some notion of a solution of \eqref{gmhd} could be defined in a weak sense past the singularity time (see, e,g, \cite{cho}). 
\begin{theorem}
\label{energyconserved}
Consider the system \eqref{gmhd} with the Dirichlet boundary condition \eqref{dbc} or the periodic boundary condition \eqref{pbc}. Assume the initial data $(u_0(x),b_0(x))$ is smooth and $u_0'(\fa)$ attains its absolute minimum, $m_0<0$, at some $\fa_0\in[0,1]$. Then
\begin{enumerate}
\item For $(\la,\ka)\in\{-1/2\}\times(-\infty,1/2]$, $\ka\neq 0$, and $\fa=\fa_0$ a zero of order $m\geq2$ of $b_0$, i.e., at least $b_0$ and $b_0'$ vanish at $\fa=\fa_0$, there exists a finite time $T>0$ such that
\begin{align}
\label{ucons1} 
\lim_{t\nearrow T}u_x(\ga(\fa_0,t),t)=-\infty
\end{align}
and
\begin{align}
\label{bcons1} 
\lim_{t\nearrow T}\bigg|\left(\frac{\partial^mb}{\partial x^m}(x,t)\right)\circ\ga(\fa_0,t)\bigg|=\infty.
\end{align}
\item For $(\la,\ka)=(-1/2,0)$ and $b_0'(\fa_0)=0$ (with $b_0(\fa_0)$ not necessarily zero), there exists a finite time $T>0$ such that 
\begin{align}
\label{ucons2} 
\lim_{t\nearrow T}u_x(\ga(\fa_0,t),t)=-\infty.
\end{align}
Additionally, $b(\ga(\fa_0,t),t)=b_0(\fa_0)$ for all $0\leq t\leq T$ and, for $b_0''(\fa_0)\neq0$,      
\begin{align}
\label{bcons2} 
\lim_{t\nearrow T}\bigg|b_{xx}(\ga(\fa_0,t),t)\bigg|=\infty.
\end{align}
\item Let $T>0$ be the blowup time in \eqref{ucons1}-\eqref{bcons2} for $(\la,\ka)\in\{-1/2\}\times(-\infty,1/2]$. Then, under the Dirichlet boundary condition \eqref{dbc}, 
\begin{align}
\label{coro1} 
\|u(\cdot,t)\|_{\infty}\leq C\|u_x(\cdot,t)\|_2\leq C
\end{align}
and 
\begin{align}
\label{coro2} 
\|b(\cdot,t)\|_{\infty}\leq C\|b_x(\cdot,t)\|_2\leq C
\end{align}
for all $0\leq t\leq T$ and positive constants $C$ that may change value from line to line. 
\item Let $T>0$ be the blowup time in \eqref{ucons1}-\eqref{bcons2} for $(\la,\ka)\in\{-1/2\}\times(-\infty,1/2]$ with the periodic boundary condition \eqref{pbc}. If in addition to $u_0'(\fa)$ attaining its absolute minimum $m_0<0$ at $\fa_0$, we choose $u_0(x)$ to have mean zero over $[0,1]$ and assume $p$ in \eqref{gmhd}-iii) satisfies the boundary condition $p(0,t)=p(1,t)$, then 
 \begin{align}
\label{coro3} 
\|u(\cdot,t)\|_{\infty}\leq C\|u_x(\cdot,t)\|_2\leq C
\end{align}
\end{enumerate}
for all $0\leq t\leq T$.
\end{theorem} 
\begin{proof} 
For $z(t)=u_x(\ga(\fa_0,t),t)$, $\fa_0\in[0,1]$, and $\la=-1/2$, we use \eqref{lagrangian} to rewrite \eqref{gmhd}-i) as
\begin{align}
\label{cons0}
z'(t)+\frac{1}{2}z^2(t)-\left(\frac{1}{2}b_x^2+\ka bb_{xx}\right)\circ\ga(\fa_0,t)=I(t).
\end{align}
Suppose $\fa_0$ is such that, at least, $b_0$ and $b_0'$ vanish at $\fa_0$. Then Lemma \ref{lemma:order} implies that $b(\ga(\fa_0,t),t)\equiv 0$ and $b_x(\ga(\fa_0,t),t)\equiv0$ for as long as solutions exist, which reduces \eqref{cons0} to 
\begin{align}
\label{cons1} 
z'(t)+\frac{1}{2}z^2(t)=I(t).
\end{align}
Let $\ka\in(-\infty,1/2]\backslash\{0\}$; we treat the case $\ka=0$ separately. Then $\la=-1/2$ and Lemma \ref{lemma:energyconserved} implies that the nonlocal term 
$$I(t)=\left(\ka-\frac{1}{2}\right)\|b_x\|_2^2-\frac{1}{2}\|u_x\|_2^2$$
satisfies  
$$-\frac{1}{2}\e_0\leq I(t)\leq\left(\ka-\frac{1}{2}\right)\e_0,$$
for $\ka\geq0$, or  
$$\left(\ka-\frac{1}{2}\right)\e_0\leq I(t)\leq-\frac{1}{2}\e_0$$
for $\ka<0$. In the above, $\e_0=\e(0)\in\R^+$ for $\e(t)$ defined in \eqref{energy}. Using these bounds on \eqref{cons1} yields
\begin{align}
\label{cons2} 
z'(t)\leq -\frac{1}{2}z^2(t)
\end{align}
for $\ka\in(-\infty,1/2]\backslash\{0\}$. Suppose $u_0'(\fa)$ attains its absolute minimum $m_0<0$ at $\fa=\fa_0$. Then solving \eqref{cons2} gives
\begin{align}
\label{cons3} 
\frac{1}{m_0}\left(1+\frac{m_0}{2}t\right)\leq \frac{1}{z(t)}< 0,
\end{align}
which implies that 
\begin{align}
\label{cons4} 
\lim_{t\nearrow T}z(t)=-\infty
\end{align}
for $T=-\frac{2}{m_0}$, which proves \eqref{ucons1}. Next, \eqref{cons3} and \eqref{jacobian} imply that
$$0<\ga_{\fa}(\fa_0,t)\leq\left(1+\frac{m_0}{2}t\right)^2,$$
which yields
\begin{align}
\label{jac0} 
\lim_{t\nearrow T}\ga_{\fa}(\fa_0,t)=0.
\end{align}
The finite-time blowup in \eqref{bcons1} now follows from \eqref{jacorder} in Lemma \ref{lemma:order}, $b_0^{(m)}(\fa_0)\neq 0$, $m\geq 2$, $\ka\in(-\infty,1/2]\backslash\{0\}$, and \eqref{jac0}.

For part 2), note that when $\ka=0$ we no longer need to use Lemma \ref{lemma:order} and the assumption that $b_0(\fa_0)=0$ to dispose of the term $\ka bb_{xx}\circ\ga(\fa_0,t)$ in \eqref{cons0}. However, we still require $b_0'(\fa_0)=0$ to eliminate the $\frac{1}{2}b_x^2\circ\ga(\fa_0,t)$ term via the same Lemma. Thus, suppose $b_0'(\fa_0)=0$ for $\ka=0$ and $\la=-1/2$. Then, using \eqref{k>=0} on \eqref{cons0} yields \eqref{cons1}, from which \eqref{cons4} once again follows. Lastly, note that for $\ka=0$, equation \eqref{gmhd}-ii) reduces to the homogeneous transport equation
$$b_t+ub_x=0,$$
whose solution along Lagrangian trajectories satisfying the IVP \eqref{lagrangian} is given by
$$b(\ga(\fa,t),t)=b_0(\fa).$$
Differentiating the above with respect to $\fa$ then gives  
$$\left(b_x\circ\ga\right)\cdot\gafa=b_0'(\fa)$$
or, equivalently,  
$$b_x(\ga(\fa,t),t)=b_0'(\fa)\cdot\ga_{\fa}^{-1}(\fa,t).$$
Differentiating the above one more time and then setting $\fa=\fa_0$ yields
$$b_{xx}(\ga(\fa_0,t),t)=b_0''(\fa_0)\cdot\ga_{\fa}^2(\fa_0,t),$$
where we used $b_0'(\fa_0)=0$. Assuming $b_0''(\fa_0)\neq0$\,\footnote{Which implies that $b_0$ has a local extrema at $\fa_0$ if $\fa_0$ happens to be an interior point of $[0,1]$.} and using \eqref{jac0} yields \eqref{bcons2}. The last part of part 2) is a direct consequence of Lemma \ref{lemma:order} if $b_0''(\fa_0)=0$.

For the last two parts of the Theorem, first note that under the Dirichlet boundary condition \eqref{dbc}, Lemma \ref{lemma:energyconserved} implies that 
$$
|u(x,t)| = \left|\int_0^x u_x(y,t)\,dy \right| \le \|u_x(\cdot,t)\|_2\leq C, 
$$
from which \eqref{coro1} follows. In the case of the periodic boundary condition \eqref{pbc}, suppose $p$ in \eqref{gmhd}-iii) satisfies the boundary condition 
\begin{align}
\label{p} 
p(0,t)=p(1,t)
\end{align}
and $u_0(x)$ has zero mean, i.e.
\begin{align}
\label{meanzero1} 
\int_0^1u_0(x)\,dx=0.
\end{align}
Then, integrating \eqref{gmhd}-iii) with respect to $x$ between $0$ and $1$ and using \eqref{p} and \eqref{meanzero1} implies that
\begin{align}
\label{meanzero2} 
\int_0^1u(x,t)\,dx=0.
\end{align}
If we now write 
$$
u(x,t) =\sum_{k} \widehat{u}_k(t)\,e^{ix k}, \qquad
\widehat{u}_k(t) =\int_0^1 e^{-ikx}\, u(x,t)\,dx,
$$
we have that
$$
\|u(\cdot,t)\|_{L^\infty} \le C\, \left[\sum_{k\not =0}|k|^2 \,|\widehat{u}_k(t)|^2\right]^{1/2} = C\, \|u_x(\cdot,t)\|_{L^2}\leq C
$$
by \eqref{meanzero2} and Lemma \ref{lemma:energyconserved}. This finishes the proof of the Theorem.
\end{proof}

\section{Finite-time Blowup for $(\la,\ka)\in\R\backslash\{0\}\times\{-\la\}$}
\label{sec:pje}

In this section, we establish finite-time blowup of solutions to \eqref{gmhd} under Dirichlet \eqref{dbc} or periodic \eqref{pbc} boundary conditions for $\ka=-\la$ and $\la\in\R\backslash\{0\}$. In particular, we exploit the existence of blowup solutions to a closely related Euler-type equation established in \cite{Sarria01} (and refined in \cite{Sarria0}), along with uniqueness of solution to an IVP for a second-order linear ODE. 

First, note that using the calculus identity
$$\frac{yy''-(1+\la)\left(y'\right)^2}{y^2}=-\frac{y^{\la}}{\la}\frac{d^2}{dt^2}\left(y^{-\la}\right)$$
on \eqref{concavity0} and then setting
\begin{align}
\label{omega}
\om(\fa,t)=\gamma_{\fa}(\fa,t)^{-\la}
\end{align}
in the resulting equation gives
\begin{align}
\label{omegaode}
\om_{tt}+\la\left[\left(\ka bb_{xx}-\la b_x^2\right)\circ \gamma+I(t)\right]\om=0
\end{align}
for $\la\in\R\backslash\{0\}$ and $I(t)=(\la+\ka)\|b_x\|_2^2-(1+\la)\|u_x\|_2^2$.

Since 
\begin{align}
\label{omega'}
\om_t(\fa,t)=-\la\om(\fa,t)\cdot u_x(\ga(\fa,t),t),
\end{align}
we complement \eqref{omegaode} with the initial conditions $\om(\fa,0)=1$ and $\om_t(\fa,0)=-\la u_0'(\fa)$. 

Setting $\la=-\ka$ and $\fa=\fa_0\in[0,1]$ in equation \eqref{omegaode}, for $\fa_0$ such that $b_0(\fa)$ has a zero of order $m\geq 2$ at $\fa_0$, yields 
\begin{align}
\label{omegaodereduced}
\om''(t)+\la I(t)\om(t)=0
\end{align}
for $\la\neq 0$, 
\begin{align}
\label{Ireduced}
I(t)=-(1+\la)\|u_x(\cdot,t)\|_2^2
\end{align}
and $\om(t)=\om(\fa_0,t)$ satisfying the initial conditions 
\begin{align}
\label{cond}
\om(0)=1,\qquad \om'(0)=-\la u_0'(\fa_0).
\end{align}
To obtain \eqref{omegaodereduced} we also used Lemma \ref{lemma:order}, which implies that $\left(\ka bb_{xx}-\la b_x^2\right)\circ \gamma(\fa_0,t)\equiv0$ since both $b_0(\fa)$ and $b_0'(\fa)$ are zero at $\fa_0$. 

Next we discuss the main ingredients in the derivation of a general solution formula for the IVP \eqref{omegaodereduced}-\eqref{cond} (or equivalently, \eqref{omegaodereduced2}-\eqref{cond2}). We start with the following proposition. 
\begin{proposition}
\label{rewrite}
Set $\om(t)=\om(\fa_0,t)$ for $\fa_0\in[0,1]$ a zero of order $m\geq2$ of $b_0(\fa)$. Then for $\la\neq 0$, the ODE \eqref{omegaodereduced} is equivalent to 
\begin{align}
\label{omegaodereduced2}
\om''(t)-\frac{1+\la}{\la}\left(\int_0^1\left(\om(\fa,t)\right)^{-2-\frac{1}{\la}}\om_t^2(\fa,t)\,d\fa\right) \om(t)=0
\end{align}
with the initial conditions 
\begin{align}
\label{cond2}
\om(0)=1,\qquad \om'(0)=-\la u_0'(\fa_0).
\end{align}
\end{proposition}
\begin{proof}
Using \eqref{gafa} and the fact that the solution $\ga(\fa,t)$ of the IVP \eqref{lagrangian} is a diffeomorphism of the unit interval for as long as a solution $u$ of \eqref{gmhd} with boundary conditions \eqref{dbc} or \eqref{pbc} exists, we have that    
$$\|u_x(\cdot,t)\|^2_2=\int_0^1u_x^2\, dx=\int_0^1\left(u_x^2\circ\gamma\right)\cdot\gafa\,d\fa=\int_0^1\ga_{\fa t}^2\cdot\ga_{\fa}^{-1}\,d\fa.$$
This and \eqref{omega} allows us to write $I(t)$ in \eqref{Ireduced} in terms of $\om(\fa,t)$ as
\begin{align}
\label{Ireduced2}  
I(t)=-\frac{1+\la}{\la^2}\int_0^1\left(\om(\fa,t)\right)^{-2-\frac{1}{\la}}\om_t^2(\fa,t)\,d\fa
\end{align}
for $\la\neq 0$. Substituting \eqref{Ireduced2} in \eqref{omegaodereduced} yields \eqref{omegaodereduced2}.
\end{proof}
Now consider the auxiliary IVP   
\begin{equation}
\label{pje}
U_{xt}+UU_{xx}=\la U_x^2-(1+\la)\|U_x(\cdot,t)\|_2^2,
\end{equation}
with smooth  initial condition $U(x,0)=U_0(x)$, parameter $\la=\frac{1}{n-1},\,\,n\in\mathbb{Z}^+,\,\, n\geq 2$, and solutions satisfying either Dirichlet or periodic boundary conditions. 

Equation \eqref{pje} is obtained by imposing on the $n-$dimensional, incompressible Euler equations 
\begin{equation}
\label{Euler}
\bfU_t+\bfU\cdot\nabla\bfU=-\grad P,\qquad \nabla\cdot\bfU=0,
\end{equation}
the stagnation-point form velocity field (\cite{Saxton1})
\begin{equation}
\label{Euleransatz}
\bfU(x,\bfx',t)=\left(U(x,t),-\frac{\bfx'}{n-1}U_x(x,t)\right),\qquad \bfx'=(x_2,x_3,\cdots,x_n).
\end{equation}
Note that \eqref{pje} can also be obtained from \eqref{gmhd} by setting $b\equiv0$ and $u=U$. Thus, we may refer to 
\eqref{pje} as the Euler analogue of the $n-$dimensional MHD-related system \eqref{gmhd}.

If one now defines the Euler-associated Lagrangian trajectory $\beta$ by
$$\beta_t=U(\beta(\fa,t),t),\qquad \beta(\fa,0)=\fa,$$
then using the same procedure that led to \eqref{omegaode} on this equation instead, gives 
\begin{align}
\label{fullomegaode}
y_{tt}(\fa,t)+H(t)y(\fa,t)=0,\qquad y(\fa,0)=1,\,\, y_t(\fa,0)=-\la U_0'(\fa),
\end{align}
for $\la\in\R\backslash\{0\}$, 
\begin{align}
\label{Ieuler}
H(t)=-\frac{1+\la}{\la}\int_0^1\left(y(\fa,t)\right)^{-2-\frac{1}{\la}}y_t^2(\fa,t)\,d\fa,
\end{align}
and     
\begin{align}
\label{eulerjac}
y(\fa,t)=\left(\beta_{\fa}(\fa,t)\right)^{-\la}=e^{-\la\int_0^tU_x(\beta(\fa,s),s)\,ds}.
\end{align}
Using a reduction of order argument and the conservation-of-mean property 
\begin{align}
\label{4}
\int_0^1\beta_{\fa}(\fa,t)\,d\fa\equiv1
\end{align}
for the corresponding Euler-related Jacobian $\beta_{\fa}$, the authors in \cite{Sarria0} derived the following general solution to the IVP \eqref{fullomegaode} (the reader may refer to \cite{Sarria0} for details).\footnote{The property \eqref{4} is a consequence of the uniqueness of solution to \eqref{lagrangian} and either Dirichlet or periodic boundary conditions; it is also a property possessed by the Jacobian $\ga_{\fa}$ in \eqref{jacobian}.} 

Set
\begin{align}
\label{phi0}  
y(\fa,t)=\phi(t)J(\fa,t),\quad J(\fa,t)=1-\la\tau(t)U_0'(\fa)
\end{align}
for 
\begin{align}
\label{phi01}  
\phi(t)=\left(\int_0^1\frac{d\fa}{J(\fa,t)^{1/\la}}\right)^{\la},\quad\la\neq 0
\end{align}
and the auxiliary time variable $\tau(\cdot)$ satisfying the IVP
\begin{align}
\label{blowtime0}  
\tau'(t)=\phi(t)^{-2},\quad \tau(0)=0.
\end{align}
It is not difficult to check that the formulae \eqref{phi0}-\eqref{blowtime0} solves the IVP \eqref{fullomegaode} for any fixed $\fa\in[0,1]$. Indeed, a straightforward computation shows that, for $\la\neq0$,  
\begin{align}
\label{phi1}  
\frac{\phi''(t)}{\phi(t)}=\frac{1+\la}{\la}\int_0^1{\left(y(\fa,t)\right)^{-2-\frac{1}{\la}}y_t^2(\fa,t)\,d\fa},
\end{align}
and so    
\begin{equation}
\label{ah}  
\begin{split}
\frac{y_{tt}(\fa,t)}{y(\fa,t)}&=\frac{2\phi'(t)J_t(\fa,t)}{\phi(t)J(\fa,t)}
+\frac{J_{tt}(\fa,t)}{J(\fa,t)}
+\frac{\phi''(t)}{\phi(t)}
\\
&=
-\frac{2\la U_0'(\fa)\phi^{-3}(t)\phi'(t)}{J(\fa,t)}+\frac{2\la U_0'(\fa)\phi^{-3}(t)\phi'(t)}{J(\fa,t)}+\frac{1+\la}{\la}\int_0^1{\left(y(\fa,t)\right)^{-2-\frac{1}{\la}}y_t^2(\fa,t)\,d\fa}
\\
&
=\frac{1+\la}{\la}\int_0^1{\left(y(\fa,t)\right)^{-2-\frac{1}{\la}}y_t^2(\fa,t)\,d\fa}
\end{split}
\end{equation}
Moreover, $y(\fa,0)=1$ and $y_t(\fa,0)=-\la U_0'(\fa)$. Formulas for $U_x(\beta(\fa,t),t)$ and $\beta_{\fa}(\fa,t)$ can be obtained from \eqref{eulerjac}, and \eqref{phi0}-\eqref{blowtime}.

At this point, we return to our IVP \eqref{omegaodereduced2}-\eqref{cond2} and show its connection to the auxiliary IVP \eqref{fullomegaode} and its general solution \eqref{phi0}-\eqref{blowtime0}. Setting $\fa=\fa_0$ in \eqref{fullomegaode} for some $\fa_0\in[0,1]$ and choosing $U_0(x)=u_0(x)$, we see that $\om(t)=\om(\fa_0,t)$, the solution of our IVP \eqref{omegaodereduced2}-\eqref{cond2}, and $y(t)=y(\fa_0,t)$, the solution of the auxiliary IVP \eqref{fullomegaode} with $\fa=\fa_0$, both solve the same ODE
\begin{align}
\label{gen0}
g''(t)+G(t)g(t)=0,
\end{align}
for $\la\neq0$ and $G(t)$ given by
\begin{align}
\label{gen1}
G(t)=-\frac{1+\la}{\la}\left(\int_0^1\left(g(\fa,t)\right)^{-2-\frac{1}{\la}}g_t^2(\fa,t)\,d\fa\right),
\end{align}
with the same initial condition 
\begin{align}
\label{gen2}
g(0)=1,\quad g'(0)=a
\end{align}
for $a=-\la u_0'(\fa_0)=-\la U_0'(\fa_0)$. By uniqueness of solution to \eqref{gen0}-\eqref{gen2}, we must have that $\om(t)=y(\fa_0,t)$. Thus, we have established the following proposition.
\begin{proposition} 
\label{Eulersolution}
Consider the system \eqref{gmhd} with the Dirichlet boundary condition \eqref{dbc} or the periodic boundary condition \eqref{pbc}, smooth initial data $(u_0(x),b_0(x))$ and $(\la,\ka)\in\R\backslash\{0\}\times\{-\la\}$. Suppose $\fa=\fa_0$ is a zero of order $m\geq 2$ of $b_0(\fa)$, i.e., at least $b_0(\fa)$ and $b_0'(\fa)$ vanish at $\fa=\fa_0$, and set 
\begin{align}
\label{Lfinal}
\ebarl_i(t)=\int_0^1{\el_i(\alpha, t)\,d\alpha},\qquad\el_i(\alpha, t)=\frac{1}{J(\alpha,t)^{i+\frac{1}{\la}}},\qquad J(\fa,t)=1-\la\tau(t)u_0'(\fa)
\end{align}
for $i=0,1,$ and $\tau(\cdot)$ the solution of the IVP
\begin{align}
\label{blowtime}  
\tau'(t)=\left(\int_0^1\frac{d\fa}{J(\fa,t)^{1/\la}}\right)^{-2\la},\quad \tau(0)=0.
\end{align}
Then 
\begin{align}
\label{omegafinal1}
\om(t)=\left(\ebarl_0(t)\right)^{\la}J(\fa_0,t),\qquad 
\om(\fa,t)=\left(\ebarl_0(t)\right)^{\la}J(\fa,t)
\end{align}
is the general solution of \eqref{omegaodereduced2}-\eqref{cond2}. Moreover,
\begin{equation}
\label{eulerjacandu_xfinal}
\ga_{\fa}(\fa_0,t)=\el_0(\fa_0,t)/\ebarl_0(t),\qquad u_x(\ga(\fa_0,t),t)=\frac{\left(\ebarl_0(t)\right)^{-2\mu}}{\la\tau(t)}\left(\frac{1}{J(\fa_0,t)}-\frac{\ebarl_1(t)}{\ebarl_0(t)}\right).
\end{equation}
\end{proposition}
From \eqref{Lfinal} and \eqref{eulerjacandu_xfinal}-ii), note that finite-time blowup of $u_x(\ga(\fa_0,t),t)$ will depend on the interplay between the nonlocal terms $\ebarl_0(t)$ and $\ebarl_1(t)$, with $1/J(\fa_0,t)$. Set
\begin{equation}
\label{eta^*}
\begin{split}
\tau^*=
\begin{cases}
\frac{1}{\lambda M_0},\,\,\,\,\,\,\,&\lambda>0,
\\
\frac{1}{\lambda m_0},\,\,\,\,\,\,\,&\lambda<0,
\end{cases}
\end{split}
\end{equation}
for $M_0>0$ and $m_0<0$, the absolute maximum and absolute minimum, respectively, of $u_0'(\fa)$ in $[0,1]$. Then, as $\tau\nearrow \tau_*,$ the term $1/J(\fa_0,t)$ in \eqref{eulerjacandu_xfinal}-ii) will diverge for certain choices of $\alpha_0$ and not at all for others. Specifically, if $\lambda>0$ and $\fa_0$ is such that such that $u_0'(\fa_0)=M_0$, then   
\begin{equation*}
\begin{split}
J(\fa_0,t)^{-1}=(1-\lambda M_0\tau(t))^{-1}\to+\infty\,\,\,\,\,\,\,\text{as}\,\,\,\,\,\,\,\tau\nearrow\tau^*=\frac{1}{\lambda M_0}.
\end{split}
\end{equation*}
Similarly, if $\lambda<0$ and $\fa_0$ is such that $u_0'(\fa_0)=m_0$, then 
\begin{equation*}
\begin{split}
J(\alpha_0,t)^{-1}=(1-\lambda m_0\tau(t))^{-1}\to+\infty\,\,\,\,\,\,\,\text{as}\,\,\,\,\,\,\,\tau\nearrow\tau^*=\frac{1}{\lambda m_0}.
\end{split}
\end{equation*}
If instead $\fa_0$ is such that $u_0'(\fa_0)\neq M_0$ when $\la>0$, or $u_0'(\fa_0)\neq m_0$ for $\la<0$, then $J(\alpha_0,t)$ will remain positive and finite as $\tau\nearrow \tau^*$. However, this does not necessarily imply that $u_x(\ga(\fa_0,t),t)$ in \eqref{eulerjacandu_xfinal}-ii) stays defined as $\tau\nearrow\tau^*$ as one must still estimate the behavior of the integral terms in \eqref{eulerjacandu_xfinal}-ii). Lastly, from \eqref{blowtime} we see that the existence of a potential finite time $t_*>0$ such that
$$\lim_{t\nearrow t^*}\tau(t)=\tau^*$$ 
will depend, in turn, upon the existence of a positive, finite limit
\begin{equation}
\label{t^*}
\begin{split}
t_*\equiv\lim_{\tau\nearrow\tau_*}\int_0^{\tau}{\left(\int_0^1{\frac{d\alpha}{(1-\lambda\mu u_0^\prime(\alpha))^{\frac{1}{\lambda}}}}\right)^{2\lambda}\,d\mu}.
\end{split}
\end{equation} 
Although we could follow a more general argument to obtain blowup criteria for a relatively large class of locally Holder arbitrary initial data $u_0$ and smooth $b_0$, we instead settle for a simple choice of $u_0(\fa)$ to simplify the presentation of our results. Nonetheless, it should be clear how the argument presented below for this particular choice of $u_0$ can be extended to more general classes of smooth initial data. 

Let us consider the case $\la>0, \ka=-\la$, under the Dirichlet boundary condition \eqref{dbc}. The case $\la<0$ with Dirichlet or periodic boundary conditions follows analogously. 

Let 
\begin{equation}
\label{data}
\begin{split}
u_0(\fa)=\fa(1-\fa),\qquad0\leq\fa\leq 1,
\end{split}
\end{equation}
for $\la>0$. Then $u_0'(\fa)=1-2\fa$ attains its absolute maximum $M_0=1$ at $\fa=0$ and, by \eqref{eta^*}, we see that $\tau^*=1/\la$. Evaluating the integral term $\ebarl_0(t)$ in \eqref{Lfinal}, we obtain
\begin{equation}
\label{L_01}
\begin{split}
\ebarl_0(t)=\frac{(1-\la^2\tau(t)^2)^{-\frac{1}{\la}}}{2(\la-1)\tau(t)}\left[(\la\tau(t)-1)\left(1+\la\tau(t)\right)^{\frac{1}{\la}}+(1+\la\tau(t))\left(1-\la\tau(t)\right)^{\frac{1}{\la}}\right]
\end{split}
\end{equation}
for $0\leq \tau<\tau^*$ and $\la\in\R\backslash\{0,1\}$, while, for $\la=1$,
\begin{equation}
\label{L_02}
\begin{split}
\ebarl_0(t)=\frac{1}{2\tau(t)}\left[\ln\left(\frac{1}{\tau(t)}+1\right)-\ln\left(\frac{1}{\tau(t)}-1\right)\right].
\end{split}
\end{equation}
Note that since $\tau(0)=0$, the above formulas are defined at $t=0$ in the limit as $\tau\searrow 0$. Using \eqref{eulerjacandu_xfinal}-i), \eqref{L_01} and \eqref{L_02}, we see that for $\fa_0=0$,
\begin{equation}
\label{jacblow1}
\gafa(0,t)=\frac{2(\la-1)\tau(t)(1+\la\tau)^{1/\la}}{(1-\la\tau)\left((1+\la\tau)(1-\la\tau)^{\frac{1}{\la}-1}-(1+\la\tau)^{1/\la}\right)}
\end{equation}
for $\la\in\R^+\backslash\{1\}$, whereas
\begin{equation}
\label{jacblow2}
\gafa(0,t)=\frac{2}{(1-\tau)(\ln(2)-\ln(1-\tau))}
\end{equation}
for $\la=1$. Then, for $\tau^*-\tau>0$ small,  
\begin{equation}
\label{jacblow3}
\gafa(0,t)\sim
\begin{cases}
C(\tau^*-\tau)^{-1},\qquad &0<\la<1,
\\
C(\tau^*-\tau)^{-\frac{1}{\la}},\qquad &\la>1,
\\
-\frac{C}{(1-\tau)\ln(1-\tau)},\qquad &\la=1
\end{cases}
\end{equation}
for $C$ a generic positive constant that may change value from line to line. This implies that   
\begin{equation}
\label{jacblow4}
\lim_{\tau\nearrow\tau^*}\gafa(0,t)=\infty
\end{equation}
for $\la>0$, so that  
\begin{equation}
\label{jacblow5}
\lim_{\tau\nearrow\tau^*}\int_0^{t}{u_x(0,s)\,ds}=\infty
\end{equation}
by formula \eqref{jacobian} for the Jacobian. In \eqref{jacblow5}, we used the result that $\ga(0,t)\equiv0$ under Dirichlet boundary conditions for as long as $u$ is defined, which follows from uniqueness of solution to \eqref{lagrangian} and the Dirichlet boundary conditions. 

The limits \eqref{jacblow4}-\eqref{jacblow5} do not necessarily imply finite-time blowup of solutions to \eqref{gmhd}; we still need to determine the existence of a finite blowup time $t^*>0$, defined by \eqref{t^*}, for which 
$$\lim_{t\nearrow t^*}\tau(t)=\tau^*.$$ 

For $\tau^*-\tau>0$ small, \eqref{L_01} and \eqref{L_02} imply that 
\begin{equation}
\label{Lestimates}
\ebarl_0(t)\sim
\begin{cases}
C(\tau^*-\tau)^{1-\frac{1}{\la}},\qquad &0<\la<1,
\\
C,\qquad &\la>1,
\\
-C\ln(1-\tau),\qquad &\la=1,
\end{cases}
\end{equation}
so that, as $\tau\nearrow\tau^*$, $\ebarl_0(t)\to\infty$ for $0<\la\leq 1$, whereas $\ebarl_0(t)$ converges to a positive constant when $\la>1$. Using the above on \eqref{t^*}, we see that the blowup time $t^*$, defined by
$$t^*=\lim_{\tau\nearrow\tau^*}t(\tau)$$
for $t(\tau)$ as in \eqref{t^*}, is positive and finite for $\la>1$, while for $0<\la\leq 1$,
\begin{equation}
\label{time1}
t^*\sim\lim_{\tau\nearrow \tau^*}
\begin{cases}
C_1+\frac{C_2}{2\la-1}-\frac{C_3}{2\la-1}(\tau^*-\tau)^{2\la-1},\qquad &\la\in(0,1)\backslash\{1/2\},
\\
C_1-C_2\ln(\tau^*-\tau),\qquad &\la=1/2,
\\
C_1+(1-\tau)(C_2-\ln(1-\tau)+\ln^2(1-\tau)),\qquad &\la=1,\, \tau^*=1,
\end{cases}
\end{equation}
for $C_i,\, i=1,2,3,$ positive constants. The estimates in \eqref{time1} imply that the blowup time $t^*$ is positive and finite for $1/2<\la\leq 1$, but $t^*=+\infty$ for $0<\la\leq 1/2$. Putting these results together with \eqref{jacblow5}, we conclude that for $\la>1/2$, there exists a finite time $t^*$ for which 
\begin{equation}
\label{jacblow5.1}
\lim_{t\nearrow t^*}\int_0^{t}{u_x(0,s)\,ds}=\infty,
\end{equation}
whereas, for $0<\la\leq 1/2$, the limit in \eqref{jacblow5.1} exists for all $t>0$. In fact, \eqref{eulerjacandu_xfinal}-ii) implies that $u_x(0,t)$ stays defined for all time for $0<\la\leq 1/2$. Since the choice of $\fa_0=0$ is also tied to our choice of initial data $b_0(\fa)$, we can summarize the result above as follows: 

For $\la>1/2$, $u_0(\fa)=\ga(1-\fa)$, and smooth $b(x,0)=b_0(\fa)$ with a zero of order $m\geq 2$ at $\fa=0$, there exists a finite time $t^*>0$ such that both $u_x(0,t)$ and $\gafa(0,t)$ diverge to $\infty$ as $t\nearrow t^*$. In contrast, if $0<\la\leq 1/2$, then both $u_x(0,t)$ and $\gafa(0,t)$ exist for all $t\geq 0$. 

Next, for $\la>0$, what if $b_0(\fa)$ has a zero of order $m\geq 2$ at $\fa=\fa_0$ but $\fa_0\neq 0$, so that $u_0'(\fa_0)$ is not equal to its absolute maximum $M_0=1$? If so, then the term $\el_0(\fa_0,t)$ in \eqref{eulerjacandu_xfinal}-i) will stay positive and bounded as $\tau\nearrow\tau^*$. Then, \eqref{eulerjacandu_xfinal}-i) and \eqref{Lestimates} imply that $\gafa(\fa_0,t)\to0$ as $t\nearrow t^*$ for $0<\la\leq 1$, but $\gafa(\fa_0,t)\to C$ for $\la>1$. Let's consider the case $0<\la\leq 1$ first. 

Since $\gafa(\fa_0,t)\to0$ for $0<\la\leq 1$ and $\fa_0\neq 0$, \eqref{jacobian} implies that 
\begin{equation}
\label{jacblow6}
\lim_{t\nearrow t^*}\int_0^{t}{u_x(\ga(\fa_0,s),s)\,ds}=-\infty.
\end{equation}
Since $t^*=\infty$ for $0<\la\leq 1/2$, but $0<t^*<\infty$ for $1/2<\la\leq 1$, we conclude that for $\fa_0\neq 0$ and $1/2<\la\leq 1$, there is a finite time $t^*>0$ such that \eqref{jacblow6} holds, whereas, for $0<\la\leq 1/2$, the time integral in \eqref{jacblow6}, and actually $u_x(\ga(\fa_0,t),t)$ itself, are defined for all $t>0$. 

Next, for $\fa_0\neq 0$ and $\la>1$, $\gafa(\fa_0,t)\to C$ as $t$ approaches the finite blowup time $t^*$ that we just established for $\gafa(0,t)$. Recalling \eqref{jacobian}, note that this does not necessarily imply that $u_x(\ga(\fa_0,t),t)$ stays finite as $t\nearrow t^*$ for $\fa_0\neq 0$. Indeed, a potential singularity in $u_x(\ga(\fa_0,t),t)$ for $\fa_0\neq 0$ could simply be time-integrable, which is actually the case here as we now show. Evaluating the time integral $\ebarl_1(t)$ in \eqref{eulerjacandu_xfinal}-ii), we obtain
\begin{equation}
\label{L_1}
\begin{split}
\ebarl_1(t)=\frac{1}{2\tau(t)}\left((1-\la\tau(t))^{-\frac{1}{\la}}-(1+\la\tau(t))^{-\frac{1}{\la}}\right)
\end{split}
\end{equation}
for $\la\neq0$. Then for $\tau^*-\tau>0$ small and $\tau^*=1/\la$, we see that
\begin{equation}
\label{L_1again}
\begin{split}
\ebarl_1(t)\sim C(\tau^*-\tau)^{-1/\la}.
\end{split}
\end{equation}
Applying the estimates \eqref{L_1again} and \eqref{Lestimates} on \eqref{eulerjacandu_xfinal}-ii) implies that for $\fa_0\neq0$ and $\la>1$,
\begin{equation}
\label{u2}
\begin{split}
u_x(\ga(\fa_0,t),t)\sim-\frac{C}{(\tau^*-\tau)^{1/\la}}
\end{split}
\end{equation}
for $\tau^*-\tau>0$ small. Consequently, 
\begin{equation}
\label{u3}
\begin{split}
\lim_{\tau\nearrow\tau^*}u_x(\ga(\fa_0,t),t)=-\infty.
\end{split}
\end{equation}
Since we already established the existence of a finite time $t^*>0$ such that $\lim_{t\nearrow t^*}\tau(t)=\tau^*$ for $\la>1$, this proves finite-time blowup of $u_x(\ga(\fa_0,t),t)$ for $\la>1$ and $\fa_0\neq 0$ a zero of order $m\geq 2$ of $b_0(\fa)$.

Lastly, the behavior of $b(x,t)$ for $\la>0$ and $\ka=-\la$ follows directly from the results above and equation \eqref{jacorder} in Lemma \eqref{lemma:order}. Since we are assuming $b_0(\fa)$ has a zero of order $m\geq2$ at $\fa=\fa_0$, the Lemma implies that the first $m-1$ spatial derivatives of $b(x,t)$ along $\ga(\fa_0,t)$ are zero for as long as $u_x(\ga(\fa_0,t),t)$ is defined. As for the $m_{\text{th}}$ spatial partial of $b(x,t)$ along $\ga(\fa_0,t)$, \eqref{jacorder} and $\ka=-\la$ imply that   
\begin{equation}
\label{B1}
\frac{\partial^m b}{\partial x^m}(x,t)\circ\ga(\fa_0,t)=b^{(m)}(\fa_0)\cdot\left(\ga_{\fa}(\fa_0,t)\right)^{-\la-m}
\end{equation}
for $b^{(m)}(\fa_0)\neq 0$. If $\fa_0=0$ and $\la>0$, the limit  \eqref{jacblow4} and \eqref{B1} imply that  
\begin{equation}
\label{B3}
\lim_{t\nearrow t^*}\frac{\partial^m b}{\partial x^m}(0,t)=0
\end{equation}
where $t^*>0$ is the finite blowup time for $u_x(0,t)$ when $\la>1/2$, but $t^*=\infty$ if $0<\la\leq 1/2$. 

Similarly, using the previous results for $\gafa(\fa_0,t)$ for $\fa_0\neq 0$ and $\la>0$, we find that 
\begin{equation}
\label{B31}
\lim_{t\nearrow t^*}\bigg|\frac{\partial^m b}{\partial x^m}(x,t)\circ\ga(\fa_0,t)\bigg|=\infty
\end{equation}
where $t^*$ is positive and finite when $1/2<\la\leq 1$, but $t^*=\infty$ for $0<\la\leq 1/2$. The sign of the infinity in the blowup in \eqref{B3} is the same as the sign of $b^{(m)}(\fa_0)\neq 0$. 

The last scenario involves $\la>1$ and $\fa_0\neq 0$. In that case, we have that the Jacobian $\gafa(\fa_0,t)$ converges to a finite positive constant as $t$ approaches the finite blowup time $t^*>0$ for $u_x(\ga(\fa_0,t),t)$ in \eqref{u3}. Consequently, \eqref{B1} implies that
\begin{equation}
\label{B4}
\lim_{t\nearrow t^*}\bigg|\frac{\partial^m b}{\partial x^m}(x,t)\circ\ga(\fa_0,t)\bigg|=C
\end{equation}
for $C$ a positive constant. Theorem \ref{lambdapos} below summarizes our results so far for $\la>0$ and $\ka=-\la$. Results for the case $\la<0$ and $\ka=-\la$ are then stated in Theorem \ref{lambdaneg} without proof since the argument is similar to that used to establish Theorem \ref{lambdapos}.
\begin{theorem}
\label{lambdapos}
Consider \eqref{gmhd} with the Dirichlet boundary condition \eqref{dbc} and parameters $\la, \ka\in\R$ satisfying $\la+\ka=0$ for $\la>0$. Suppose $b(x,0)=b_0(x)$ is smooth and has a zero of order $m\geq 2$ at some fixed $\fa_0\in[0,1]$, i.e., at least $b_0(\fa)$ and $b_0'(\fa)$ vanish at $\fa=\fa_0$. Set $u_0(\fa)=\fa(1-\fa)$, so that $u_0'(\fa)=1-2\fa$ attains its absolute maximum $M_0=1$ in $[0,1]$ at $\fa=0$. Then
\begin{enumerate}
\item For $\la\in(0,1/2]$ and $\ka=-\la$, both  $u_x(\ga(\fa_0,t),t) $ and $b_x(\ga(\fa_0,t),t)$ exist globally in time.   
\item For $\la>1/2$ and $\ka=-\la$, if $\fa_0=0$, then there exists a finite time $t^*>0$ such that 
\begin{equation}
\label{t1}
\lim_{t\nearrow t^*}\int_0^tu_x(\ga(\fa_0,s),s)\,ds=\lim_{t\nearrow t^*}\int_0^tu_x(\fa_0,s)\,ds=\infty,
\end{equation}
while, if $\fa_0\neq0$, then
\begin{equation}
\label{t2}
\lim_{t\nearrow t^*}\int_0^tu_x(\ga(\fa_0,s),s)\,ds=-\infty. 
\end{equation}
\end{enumerate}
Additionally, let $t^*>0$ be the finite blowup time in \eqref{t1} and \eqref{t2} for $\la>1/2$ and $\ka=-\la$. Then 
\begin{enumerate}
\setcounter{enumi}{2}
\item The order $m\geq 2$ of the zero $\fa_0$ of $b_0(\fa)$ is preserved by $b(x,t)$ along the trajectory $\ga(\fa_0,t)$ for all $t\in[0,t^*]$.  
\item  For $\la>1/2$ and $\ka=-\la$, if $\fa_0=0$, then   
\begin{equation}
\label{t3}
\lim_{t\nearrow t^*}\frac{\partial^m b}{\partial x^m}(x,t)\circ\ga(\fa_0,t)=\lim_{t\nearrow t^*}\frac{\partial^m b}{\partial x^m}(\fa_0,t)=0.
\end{equation}
\item For $1/2<\la\leq 1$ and $\ka=-\la$, if $\fa_0\neq 0$, then 
\begin{equation}
\label{t4}
\lim_{t\nearrow t^*}\bigg|\frac{\partial^m b}{\partial x^m}(x,t)\circ\ga(\fa_0,t)\bigg|=\infty.
\end{equation}
\item For $\la>1$ and $\ka=-\la$, if $\fa_0\neq 0$, then
\begin{equation}
\label{t5}
\lim_{t\nearrow t^*}\bigg|\frac{\partial^m b}{\partial x^m}(x,t)\circ\ga(\fa_0,t)\bigg|=C
\end{equation}
for $C$ a nonzero constant. 
\end{enumerate}
\end{theorem}
\begin{theorem}
\label{lambdaneg}
Consider \eqref{gmhd} with Dirichlet boundary conditions \eqref{dbc} and parameters $\la, \ka\in\R$ satisfying $\la+\ka=0$ for $\la<0$. Suppose $b(x,0)=b_0(x)$ is smooth and has a zero of order $m\geq 2$ at some fixed $\fa_0\in[0,1]$, i.e., at least $b_0(\fa)$ and $b_0'(\fa)$ vanish at $\fa=\fa_0$. Set $u_0(\fa)=\fa(1-\fa)$, so that $u_0'(\fa)=1-2\fa$ attains its absolute minimum $m_0=-1$ in $[0,1]$ at $\fa=1$. Lastly, set
\begin{equation}
\label{min}
m(t)=\min_{\fa\in[0,1]}u_x(\ga(\fa,t),t).
\end{equation}
Then $m(t)=u_x(\ga(1,t),t)=u_x(1,t)$, and
\begin{enumerate}
\item For $\la<0$ and $\ka=-\la$, 
\begin{enumerate}
\item if $\fa_0=1$, then there exists a finite $t^*>0$ such that  
\begin{equation}
\label{min2}
\lim_{t\nearrow t^*}\int_0^tm(s)\,ds=-\infty.
\end{equation}
\item if $\fa_0\neq1$, then $u_x(\ga(\fa_0,t),t)$ is defined for all $0\leq t\leq t^*$.
\end{enumerate}
\end{enumerate}

Additionally, let $t^*$ be the finite blowup time in \eqref{min2}. Then 
\begin{enumerate}
\setcounter{enumi}{1}
\item The order $m\geq 2$ of the zero $\fa_0$ of $b_0(\fa)$ is preserved by $b(x,t)$ along the trajectory $\ga(\fa_0,t)$ for all $t\in[0,t^*]$. 

\item For $\la<0$ and $\ka=-\la$,

\begin{enumerate}
\item if $\la>-m$ and $\fa_0=1$, then  
\begin{equation}
\label{bb1}
\lim_{t\nearrow t^*}\bigg|\frac{\partial^m b}{\partial x^m}(\fa_0,t)\bigg|=\infty.
\end{equation}
\item if $\la<-m$ and $\fa_0=1$, then
\begin{equation}
\label{bb2}
\lim_{t\nearrow t^*}\frac{\partial^m b}{\partial x^m}(\fa_0,t)=0.
\end{equation}
\item if $\la=-m$ and $\fa_0=1$, then 
\begin{equation}
\label{bb3}
\frac{\partial^m b}{\partial x^m}(x,t)\circ\ga(\fa_0,t)=b_0^{(m)}(\fa_0)\neq 0
\end{equation}
for all $0\leq t\leq t^*$.
\item if $\fa_0\neq 1$, then 
\begin{equation}
\label{bb4}
\lim_{t\nearrow t^*}\frac{\partial^m b}{\partial x^m}(x,t)\circ\ga(\fa_0,t)=Cb_0^{(m)}(\fa_0)\neq 0
\end{equation}
for $C$ a positive constant.
\end{enumerate} 
\end{enumerate}
\end{theorem}

\begin{remark}
In more generality, it can be shown using a Taylor series argument that the results in Theorems \ref{lambdapos} and \ref{lambdaneg} will hold if $u_0(\fa)\in C^2([0,1])$ is such that $u_0''(\fa)\neq 0$ at the locations where $u_0'(\fa)$ attains its absolute maximum (if $\la>0$) or absolute minimum (if $\la<0$). Clearly that is the case for our present choice of $u_0(\fa)$ since $u_0''(\fa)=-2$. In the context of solutions of the form \eqref{spf} of the $n-$dimensional ideal MHD equations, this condition simply means that a non-vanishing initial ``vorticity'' $u_0''(\fa)$ at the corresponding location(s) where $u_0'(\fa)$ attains its global extrema leads to finite-time blowup, while a vanishing one will result in global-in-time solutions (\cite{Sarria01}).
\end{remark}

\begin{remark}
Part 1 of Theorem \ref{lambdapos} showed that, for a particular choice of initial data $u_0(x)$ and parameter values $(\la,\ka)\in(0,1/2]\times\{-\la\}$, both $u_x(x,t)$ and $b_x(x,t)$ exist for all time along the Lagrangian trajectory $\ga(\fa,t)$ that starts at $\fa=\fa_0$. It is currently unknown if this is also the case for other choices of initial data or even for the same choice of data if $\fa\neq\fa_0$.   
\end{remark}

\section{A Comparison Blowup Criterion}
\label{sec:comparison}

In the previous section, we exploited a construction done in \cite{Sarria0} for a general solution of the $n-$dimensional incompressible Euler equation \eqref{Euler} to obtain finite-time blowup of solutions to \eqref{gmhd} for particular values of the parameters $\la $ and $\ka$. Basically, we showed that when we restrict the solution $u_x(x,t)$ of the system \eqref{gmhd} to the family of Lagrangian trajectories
$$\Gamma:=\{\langle\ga(\fa_0,t),t\rangle\,|\,\fa_0\in[0,1]\,\text{is a zero of order}\,\, m\geq 2\,\, \text{of}\,\, b_0(\fa)\},$$
then for $\ka=-\la$ and $\la\in\R\backslash\{0\}$, the evolution of $u_x(\ga(\fa_0,t),t)$ follows that of the Euler-associated function  $U_x(\ga(\fa_0,t),t)$, i.e., the solution of \eqref{pje}. Note that the parameters values $(\la,\ka)=(1,1)$ and $(\la,\ka)=(1/2,1)$, which correspond to solutions of the form \eqref{spf} of the $2D$ and $3D$ ideal MHD equations, respectively, are missing from Theorems \ref{lambdapos} and \ref{lambdaneg}. In this section, we slightly modify a standard Sturm comparison argument to obtain blowup criteria for solutions of the form \eqref{spf} of the 2D ideal MHD system \eqref{gmhd}, i.e. 
\begin{equation}
\label{2dmhd}
\begin{split}
&u_{xt}+uu_{xx}=u_x^2-b_x^2+bb_{xx}+2\|b_x\|_2^2-2\|u_x\|_2^2,
\\
&b_{t}+ub_{x}= bu_{x},
\end{split}
\end{equation}
in terms of the MHD and Euler pressure functions for smooth initial data. In particular, Theorem \ref{comparison} below compares the solution of \eqref{2dmhd} to that of the associated 2D incompressible Euler equation  
\begin{equation}
\label{2dpje}
U_{xt}+UU_{xx}=U_x^2-2\|U_x(\cdot,t)\|_2^2,
\end{equation}
which is obtained from \eqref{pje} by setting $\la=1$. Just as in the previous section, we use $U(x,0)=U_0(x)=x(1-x)$ as the 2D Euler-associated initial data, which means that we will be working under Dirichlet boundary conditions $U(0,t)=U(1,t)=0$. Since $U_0'(x)=1-2x$ attains its absolute maximum $M_0=1$ at $x=0$, we know that (\cite{Sarria01, sarria1}) 
\begin{equation}
\label{eblow}
\begin{split}
\lim_{t\nearrow T_e}U_x(0,t)=\lim_{t\nearrow T_e}\|U_x(\cdot,t)\|_2=\infty
\end{split}
\end{equation}
for 
\begin{equation}
\label{T_e}
\begin{split}
T_e&\equiv\lim_{\tau\nearrow1}\int_0^{\tau}{\left(\int_0^1{\frac{d\alpha}{1-\mu U_0^\prime(\alpha)}}\right)^{2}\,d\mu}
\\
&=\lim_{\tau\nearrow1}\int_0^{\tau}{\left[\frac{1}{2\tau}\ln\left(\frac{1+\tau}{1-\tau}\right)\right]^{2}\,d\mu}
\\
&=\frac{\pi^2}{6}.
\end{split}
\end{equation}
It is worth noting that $t=T_e$ and $x=0$ yield the \textit{earliest} blowup time for $U_x(x,t)$ and $\|U_x(\cdot,t)\|_2$. We now state and prove the main result of this section.  

\begin{theorem}
\label{comparison}
Consider the system \eqref{2dmhd} describing solutions of the form 
$$\bfu(x,y,t)=(u(x,t),-yu_x(x,t)),\qquad \bfb(x,y,t)=(b(x,t),-yb_x(x,t))$$
of the 2D ideal MHD equations \eqref{inviscidMHD} under the Dirichlet boundary condition \eqref{dbc} and smooth initial data $(u_0(x),b_0(x))$. Moreover, let $U_0(x)=x(1-x)$ be the initial data of the associated 2D Euler equation \eqref{2dpje} whose solution $U(x,t)$ satisfies the Dirichlet boundary condition
\begin{equation}
\begin{split}
\label{eulerdbc}
U(0,t)=U(1,t)=0
\end{split}
\end{equation}
and has finite interval of existence $\I_e=(0,T_e)$ for $T_e=\pi^2/6$. More precisely (\cite{Sarria01, sarria1}),  
\begin{equation}
\label{eblow2}
\begin{split}
\lim_{t\nearrow T_e}U_x(0,t)=\lim_{t\nearrow T_e}\|U_x(\cdot,t)\|_2=\infty.
\end{split}
\end{equation} 
Lastly, suppose $u_0(x)$ and $b_0(x)$ are such that   
\begin{equation}
\begin{split}
\label{initial}
b_0'(0)=0,\quad u_0'(0)\geq U_0'(0),\quad\text{and}\quad  \|U_0'(x)\|^2_2\geq\|u_0'(x)\|^2_2-\|b_0'(x)\|^2_2. 
\end{split}
\end{equation}
If 
\begin{equation}
\begin{split}
\label{ass}
\|U_x(\cdot,t)\|^2_2\geq\|u_x(\cdot,t)\|^2_2-\|b_x(\cdot,t)\|^2_2
\end{split}
\end{equation}
for all $t\in \I_e$, then
\begin{equation}
\begin{split}
\label{2dmhdblow}
\lim_{t\nearrow T_e}u_x(0,t)=\infty.
\end{split}
\end{equation}
\end{theorem}
\begin{proof}
Since $b_0'(0)=0$, differentiating \eqref{2dmhd}-ii) with respect to $x$, setting $x=0$, and using the Dirichlet boundary condition \eqref{dbc} yields $b_{xt}(0,t)=0$, so that 
\begin{equation}
\begin{split}
\label{bzero}
b_x(0,t)=b_0'(0)=0.
\end{split}
\end{equation}
Set $z(t)=u_x(0,t)$ and $w(t)=U_x(0,t)$. Then evaluating \eqref{2dmhd}-i) and \eqref{2dpje} at $x=0$ and using \eqref{dbc} and \eqref{eulerdbc} gives 
\begin{equation}
\begin{split}
\label{bdry}
z'(t)=z^2(t)+2\|b_x\|_2^2-2\|u_x\|_2^2\,,\quad w'(t)=w^2(t)-2\|U_x\|_2^2\,.
\end{split}
\end{equation}
Setting $\sigma(t)=z(t)-w(t)$ and subtracting the equations in \eqref{bdry} yields
\begin{equation}
\begin{split}
\label{sigma}
\sigma'=
(z+w)\sigma-2\h(t)
\end{split}
\end{equation}
for 
$$\h(t)=\|u_x\|_2^2-\|b_x\|_2^2-\|U_x\|_2^2.$$
Then, using \eqref{initial}-iii) and \eqref{ass} on \eqref{sigma} implies that
\begin{equation}
\begin{split}
\label{sigma2}
\sigma'\geq (z+w)\sigma
\end{split}
\end{equation}
for all $t\in\I_e$. An application of Gronwall's inequality on \eqref{sigma2} and \eqref{initial}-ii) now gives
$$\sigma(t)\geq \sigma(0)e^{\int_0^t(z(s)+w(s))\,ds}\geq 0$$
or, equivalently,
$$u_x(0,t)\geq U_x(0,t)$$
for all $0\leq t<T_e$. Letting $t\nearrow T_e$ in the above and using \eqref{eblow}-i) yields \eqref{2dmhdblow}. 
\end{proof}
The assumption \eqref{ass} in Theorem \ref{comparison} can be stated in terms of the corresponding nonlocal pressure terms as follows. For $0\leq x\leq 1$ and $y>0$, let $P^E(x,y,t)$ and $p^M(x,y,t)=P^M+\frac{1}{2}\bfb^2$ denote the scalar 2D Euler hydrodynamic pressure and the scalar 2D MHD total pressure, respectively, for $P^M$ the MHD hydrodynamic pressure. Then for solutions of the form \eqref{spf} and \eqref{Euleransatz}, it follows that the nonlocal terms 
$$J(t)=-2\|U_x\|_2^2<0$$
in the 2D Euler-associated equation \eqref{pje} (for $\la=1$) and 
$$I(t)=2\|b_x\|_2^2-2\|u_x\|_2^2$$ in \eqref{gmhd}-iv (for $\la=\ka=1$) satisfy
$$J(t)=\frac{1}{y}P_y^E<0,\qquad I(t)=\frac{1}{y}p^M_y=\frac{1}{y}P^M_y+b_x^2,$$ 
where, by \eqref{eblow2},
$$\lim_{t\nearrow T_e}P_y^E(t)=-\infty$$
for $T_e=\pi^2/6$. Consequently, the assumption \eqref{ass} is equivalent to 
$$\frac{1}{y}P^E_y\leq\frac{1}{y}P^M_y+b_x^2$$
for all $t\in(0,T_e)$, which doesn't necessarily imply blowup in the MHD pressure $I(t)$.

Additional results may be established using more sophisticated comparison results (see, e.g., \cite{swanson, preston1, preston2}), but Theorem \ref{comparison} is sufficient for our purposes of exploring the role played by the Euler and MHD pressure terms in the blowup of infinite energy solutions to \eqref{2dmhd} of the type \eqref{spf} when blowup of the infinite energy solutions of the Euler-associated equation \eqref{2dpje} is already known.

\section{Blowup Suppression for $(\la,\ka)=(-1/2,0)$}
\label{sec:suppressing}

All finite-time blowup results and blowup criteria presented so far require the existence of $\fa_0\in[0,1]$ such that, at least, the zeroth and first derivative of $b(x,0)=b_0(x)$ are zero there, i.e. 
\begin{align}
\label{vanish}
b_0(\fa_0)=b_0'(\fa_0)=0.
\end{align}
If \eqref{vanish} holds, then we have shown that $u_x\circ\ga(\fa_0,t)$ and some spatial derivative of order $m\geq2$ of $b(x,t)$ along $\ga(\fa_0,t)$ blowup in finite time for various values of the parameters $\la$ and $\ka$. In contrast, Lemma \ref{lemma:order} implies that if \eqref{vanish} holds, then $b(x,t)$ and $b_x(x,t)$ along $\ga(\fa_0,t)$ are identically zero for as long as solutions exist. If we interpret this in the MHD setting, this would imply that the magnetic field $\bfb(x,\bfx',t)=(b(x,t),-\frac{\bfx'}{n-1}b_x(x,t))$ in \eqref{spf}-ii) is identically  zero along this family of trajectories. This all appears to indicate a particular blowup mechanism for solutions of \eqref{gmhd} that favors the vanishing of $b$ and/or $b_x$.\footnote{Since setting $\bfb\equiv0$ reduces \eqref{inviscidMHD} to the incompressible Euler equations, one could argue that these blowup results are somewhat Euler in nature.} However, discussion of blowup only in the velocity gradient of the MHD equations and/or the role that a zero of the magnetic field may play in finite time blowup is not new. Yan (\cite{Yan}) showed that finite-time blowup of infinite energy, self-similar solutions of the 3D incompressible MHD equations \eqref{MHD} occurs in the velocity gradient but not the magnetic field. Moreover, numerical results by Gibbon and Ohkitani (\cite{gibbon3}) indicate finite-time blowup of infinite energy, periodic solutions of the 3D ideal MHD equations \eqref{inviscidMHD} of the form \eqref{2andahalf} in the velocity gradient along with a simultaneous, yet late-stage, blowup in the magnetic field. The authors, however, point out that the blowup in the magnetic field would be suppressed if the magnetic field is initially zero at the blowup location, which has precisely been our situation in all blowup results established so far in this paper. In Theorem \ref{global} below, we show that for values of $\la$ and $\ka$ for which there is finite-time blowup of $u_x(\ga(\fa_0,t),t)$ under the condition \eqref{vanish} (see part 2 of Theorem \ref{energyconserved}, or Theorem \ref{concave} ), removing the assumption $b_0'(\fa_0)=0$ results in the suppression of said blowup as long as the initial data satisfies a simple smallness condition. 
\begin{theorem}
\label{global}
Consider the system \eqref{gmhd} with smooth initial data $(u_0(x),b_0(x))$ and either the Dirichlet boundary condition \eqref{dbc} or the periodic boundary condition \eqref{pbc}. Suppose there is $\fa_0\in[0,1]$ such that $b_0(\fa_0)=0$ and $b_0'(\fa_0)>0$, and that the nontrivial initial data satisfies the smallness condition
\begin{equation}
\label{smallness}
\begin{split}
\|u_0'(x)\|_2^2+\|b_0'(x)\|_2^2\leq 1.
\end{split} 
\end{equation}
Then, for $(\la,\ka)=(-1/2,0)$, neither $u_x(\ga(\fa_0,t),t)$ nor $b_x(\ga(\fa_0,t),t)$ blows up in finite time. 
\end{theorem}
\begin{proof}
Suppose $b_0(\fa_0)=0$ and $b_0'(\fa_0)>0$ for some $\fa_0\in[0,1]$. By Lemma \ref{lemma:order}, this implies that $b(\ga(\fa_0,t),t)\equiv0$ and $b_x(\ga(\fa_0,t),t)>0$ for as long as a solution exists. Moreover, since $\la=-1/2$ and $\ka=0$, Lemma \ref{lemma:energyconserved} implies that the nonlocal term $I(t)$ in \eqref{gmhd}-iv) is conserved, i.e.
\begin{equation}
\label{Iglobal}
\begin{split}
I(t)=-\frac{1}{2}\left(\|u_x\|_2^2+\|b_x\|_2^2\right)=-\frac{1}{2}\e(t)=-c^2
\end{split} 
\end{equation}
for $c^2=\frac{1}{2}\e_0\in\R^+$, $\e_0=\e(0)$ and $\e(t)=\|u_x\|_2^2+\|b_x\|_2^2$. Set 
$$z(t)=u_x(\ga(\fa_0,t),t),\quad w(t)=b_x(\ga(\fa_0,t),t).$$
Then we can write equations \eqref{gmhd}-i)-ii) along $\ga(\fa_0,t)$ as
\begin{equation}
\label{odd}
\begin{split}
z'(t)&=\frac{1}{2}w^2(t)-\frac{1}{2}z^2(t)-c^2,
\\
w'(t)&=-w(t)z(t),
\end{split} 
\end{equation}
where solving the second equation gives 
$$w(t)=b_0'(\fa_0)\cdot e^{-\int_0^tz(s)\,ds}.$$
Since $b_0'(\fa_0)>0$, this implies that $w(t)>0$. This is also a consequence of Lemma \ref{lemma:order} and our assumption that $b_0(\fa_0)=0$ and $b_0'(\fa_0)>0$. 

Define the strictly positive function (\cite{Aconstantin0, wunsch1})
$$W(t)=w_0w(t)+\frac{w_0}{w(t)}\left(1+z^2(t)\right)$$
for $w_0=w(0)=b_0'(\fa_0)>0$. Computing $W'$ with the help of \eqref{odd} gives, after simplification,
\begin{equation}
\label{W}
\begin{split}
W'&=w_0w'-w_0w^{-2}w'\left(1+z^2\right)+2w_0w^{-1}zz'
\\
&=(1-2c^2)\frac{w_0}{w}z
\\
&=(1-\e_0)\frac{w_0}{w}z
\end{split} 
\end{equation}\
where $1-\e_0\geq 0$ due to \eqref{smallness}. Since $z\leq 1+z^2$ and $w_0/w>0$, we see that 
\begin{equation}
\label{W1}
\begin{split}
\frac{w_0}{w}z\leq w_0w+\frac{w_0}{w}\left(1+z^2\right)=W,
\end{split} 
\end{equation}
so that
$$(1-\e_0)\frac{w_0}{w}z\leq(1-\e_0)W.$$
Using the above on \eqref{W} then gives
\begin{equation}
\label{W2}
\begin{split}
W'(t)\leq (1-\e_0)W(t)
\end{split} 
\end{equation}
for $1-\e_0\geq 0$. Invoking Gronwall's inequality on \eqref{W2} now yields
$$0<W(t)\leq W(0)e^{(1-\e_0)t}$$
for
$$W(t)=b_0'(\fa_0)b_x(\ga(\fa_0,t),t)+\frac{b_0'(\fa_0)}{b_x(\ga(\fa_0,t),t)}(1+(u_x(\ga(\fa_0,t),t))^2).$$
Consequently, neither $u_x(\ga(\fa_0,t),t)$ nor $b_x(\ga(\fa_0,t),t)$ blows up in finite-time.
\end{proof}

\begin{remark}
Suppose $u_0$ and $b_0$ are both odd at $x=1/2$. Then, the structure of \eqref{gmhd} implies that if $u$ and $b$ are solutions with this initial data, then they will retain this property for as long as solutions exist. In particular, this means that $u, b, u_{xx}$ and $b_{xx}$ vanish at $x=1/2$ since they are all odd, while $u_{x}$ and $b_x$ are even. The result in Theorem \ref{global} still holds if one replaces the existence of $\fa_0$ with oddness at $x=1/2$.  
\end{remark}

\section{Global Existence for $(\la,\ka)=(0,0)$}
\label{sec:zero}
Our last result is that solutions of \eqref{gmhd}-\eqref{dbc} (or \eqref{pbc}) exist for all time if $\la=\ka=0$.
\begin{theorem}
\label{global2}
Consider the system \eqref{gmhd} with smooth initial data $(u_0(x),b_0(x))$ and either the Dirichlet boundary condition \eqref{dbc} or the periodic boundary condition \eqref{pbc}. If $\la=\ka=0$, then solutions exist globally in time. In particular, for all $t\geq 0$,  the Jacobian \eqref{jacobian} satisfies
\begin{equation}
\label{zerojacobian}
\begin{split}
\gafa(\fa,t)=e^{tu_0'(\fa)}\left(\int_0^1e^{tu_0'(\fa)}\,d\fa\right)^{-1},
\end{split} 
\end{equation}
while 
\begin{equation}
\label{zerouxbx}
\begin{split}
b(\ga(\fa,t),t)=b_0(\fa)\qquad 0\leq u_0'(\fa)-u_x(\ga(\fa,t),t)\leq \int_0^1u_0'(\fa)e^{tu_0'(\fa)}\,d\fa. 
\end{split} 
\end{equation}
\end{theorem}
\begin{proof}
First, note that setting $\la=\ka=0$ in \eqref{gmhd}-i)-ii) implies that
$$\partial_t(u_x(\ga(\fa,t),t))\leq0,\qquad \partial_t(b(\ga(\fa,t),t))=0$$
or, after integrating,
\begin{equation}
\label{solving}
\begin{split}
u_x(\ga(\fa,t),t)\leq u_0'(\fa)\qquad b(\ga(\fa,t),t)=b_0(\fa)
\end{split} 
\end{equation}
Next, setting $\lambda=\ka=0$ in \eqref{concavity0} yields 
$$\partial^2_t\left(\ln\gamma_\alpha\right)=I(t)$$
for $I(t)=-\int_0^1{u_x^2\,dx}$. Then, integrating the above twice in time and using $\gamma_{\fa t}=(u_x(\gamma,t))\cdot\gamma_\alpha$ and $\gamma_\alpha(\alpha,0)=1,$ we obtain
\begin{equation}
\label{eq:02}
\begin{split}
\gamma_\alpha(\alpha,t)=e^{tu_0^\prime(\alpha)}e^{\int_0^t{(t-s)I(s)ds}}.
\end{split}
\end{equation}
Since $\gamma_\alpha$ has mean one in $[0,1],$ we integrate (\ref{eq:02}) in $\alpha$ to find
\begin{equation}
\label{eq:03}
\begin{split}
e^{\int_0^t{(t-s)I(s)ds}}=\left(\int_0^1{e^{tu_0^\prime(\alpha)}d\alpha}\right)^{-1}.
\end{split}
\end{equation}
Combining this with (\ref{eq:02}) gives 
\begin{equation}
\label{eq:04}
\begin{split}
\gamma_{\alpha}(\alpha,t)=e^{tu_0^\prime(\alpha)}\left(\int_0^1{e^{tu_0^\prime(\alpha)}d\alpha}\right)^{-1},
\end{split}
\end{equation}
a bounded expression for $\gamma_\alpha$ which we differentiate with respect to $t$ to get
\begin{equation}
\label{eq:05}
\begin{split}
u_x(\gamma(\alpha,t),t)=u_0^\prime(\alpha)-\frac{\int_0^1{u_0^\prime(\alpha)e^{tu_0^\prime(\alpha)}d\alpha}}{\int_0^1{e^{tu_0^\prime(\alpha)}d\alpha}}.
\end{split}
\end{equation}
Using \eqref{eq:04} along with successive differentiation with respect to $\fa$ of \eqref{solving}-ii) yields bounded expressions for all spatial derivatives of $b\circ\ga$.

Lastly,  combining \eqref{solving}-i) with \eqref{eq:05} implies that
$$\partial_t\ln\left(\int_0^1 e^{tu_0'(\fa)}\,d\fa\right)\geq 0,$$
so that $\int_0^1e^{tu_0'(\fa)}\,d\fa\geq 1$. Then, \eqref{solving}-i) and \eqref{eq:05} give 
$$\int_0^1u_0'(\fa)e^{tu_0'(\fa)}\,d\fa=(u_0'(\fa)-u_x(\ga(\fa,t),t))\int_0^1e^{tu_0'(\fa)}\,d\fa\geq u_0'(\fa)-u_x(\ga(\fa,t),t)\geq 0,$$
which is \eqref{zerouxbx}-ii).
\end{proof}
It would be interesting to see if finite-time blowup of solutions to \eqref{gmhd} is possible without the assumption \eqref{vanish} for parameter values other than $(\la,\ka)=(-1/2,0)$. Moreover, global regularity for either $\ka>-\la$ for $\la\in\R$, or $\ka<-\la$ for $\la\in(-\infty,-1)\cup[0,\infty)$, has yet to be determined. Finally, in upcoming work we will be studying the effects that damping may have on the regularity of solutions of \eqref{gmhd}.

\makeatletter \renewcommand{\@biblabel}[1]{\hfill#1.}\makeatother


\begin{thebibliography}{1000}

\bibitem{bauer} M. Bauer, Y. Lu, and C. Maor, A Geometric View on the Generalized Proudman–Johnson and r-Hunter–Saxton Equations. J Nonlinear Sci 32 (2022), 17.

\bibitem{preston2} M. Bauer, S.C. Preston and J. Valletta, Liouville comparison theory for breakdown of Euler-Arnold equations, J. Differential Equations 407 (2024), 392-431.

\bibitem{Camassa1} R. Camassa and D. Holm, An integrable shallow water equation with peaked solitons, Phys. Rev. Lett. 71 (11) (1993), 1661-1664.

\bibitem{chae07} D. Chae, Global regularity for the 2D Boussinesq equations with partial viscosity terms, Adv. Math. 203 (2006), 497-513.

\bibitem{chae0} D. Chae and J. Lee, On the blow-up criterion and small data global existence for the Hall-magnetohydrodynamics, J. Differential Equations 256 (11) (2014), 3835-3858.

\bibitem{childress} S. Childress, G.R. Ierley, E.A. Spiegel and W.R. Young, Blow-up of unsteady two-dimensional Euler and Navier-Stokes solutions having stagnation-point form, J. Fluid Mech. 203 (1989), 1-22.

\bibitem{cho} C.H. Cho and M. Wunsch, Global and singular solutions to the generalized Proudman–Johnson equation, Journal of Differential Equations Vol. 249 (2)
(2010), 392-413.

\bibitem{Aconstantin0} A. Constantin and R.I. Ivanov, On an integrable two-component Camassa-Holm shallow water system. Physics Letters A 372 (2008), 7129-7132.

\bibitem{Aconstantin2} A. Constantin and D. Lannes, The hydrodynamical relevance of the Camassa-Holm and Degasperis-Procesi Equations, Arch. Rational Mech. Anal. 192 (2009), 165-186.

\bibitem{constantin2} P. Constantin, The Euler equations and nonlocal conservative Riccati equations, Int. Math. Res. Not. 9 (2000), 455-465.

\bibitem{drazin} P. Drazin and W. Reid, Hydrodynamic Stability, Cambridge University Press, 1981.

\bibitem{Dullin1} H. R. Dullin, G. A. Gottwald, D. D. Holm, Camassa-Holm, Korteweg-de Vries-5 and other asymptotically equivalent equations for shallow water waves, Fluid Dyn. Res., 33 (2003), 73-95.

\bibitem{Escher0} J. Escher, O. Lechtenfeld and Z. Yin, Well-posedness and blow up phenomena for the 2-component Camassa-Holm equations, Discrete Contin. Dyn. Syst. 19 (3) (2007), 493-513.

\bibitem{gibbon1} J.D. Gibbon, A.S. Fokas, C.R. Doering, Dynamically stretched vortices as solutions of the 3D Navier–Stokes equations, Physica D (1999), 497-510.

\bibitem{gibbon3} J.D. Gibbon, K. Ohkitani, Singularity formation in a class of stretched solutions of the equations for ideal magneto-hydrodynamics. Nonlinearity 14 (5) (2001), 1239-1264.

\bibitem{gibbon4} J. Gibbon, D. Moore, and J. Stuart, Exact, infinite energy, blow-up solutions of the three-dimensional Euler equations, Nonlinearity 16, No. 5 (2003), 1823-1831.

\bibitem{gibbon}
J. D. Gibbon, The three-dimensional Euler equations: Where do we stand?, Physica D 237, (2008), 1894-1904.

\bibitem{gill} A.E. Gill, \textsl{Atmosphere-Ocean Dynamics}, Academic Press (London) 1982.

\bibitem{holm} D. D. Holm and M.F. Staley, Wave structure and nonlinear balances in a
family of evolutionary PDEs. SIAM J. Appl. Dyn. Syst. 2 (3) (2003), 323-380.

\bibitem{hou0} T.Y. Hou and C. Li, Global well-posedness of the viscous Boussinesq equations. Discrete Contin. Dyn. Syst. 12 (2005), 1-12.

\bibitem{hou22} T. Y. Hou and C. Li, Dynamic stability of the 3D axi-symmetric Navier-Stokes equations with swirl, Comm. Pure Appl. Math., 61 (5) (2008), 661-697.

\bibitem{hou1} T.Y. Hou, Blow–up or no blow-up? Unified computational and analytic approach to 3D incompressible Euler and Navier–Stokes equations, Acta Numerica (2009), 277-346.

\bibitem{Hunter1} J.K. Hunter and R. Saxton, Dynamics of director fields, SIAM J. Appl. Math. 51(6) (1991), 1498-1521.

\bibitem{Johnson1} R.S. Johnson, Camassa-Holm, Korteweg-de Vries and related models for water waves, J. Fluid. Mech. 455 (2002), 63-82.

\bibitem{Kato} T. Kato, Quasi-linear equations of evolution, with applications to partial differential equations, Spectral Theory and Differential Equations, Lecture Notes in Math. 448, Springer Verlag, Berlin (1975), 25-70. 


\bibitem{Kim} N. Kim and B. Lkhagvasuren, Global existence for a partially linear 3D Euler flow, J. Korean Math. Soc. 55 (1) (2018), 211-224.

\bibitem{Minkyu} M. Kwak and B. Lkhagvasuren, Existence of a Class of Stretched 2D Ideal Magnetohydrodynamic Flows, J Math Sci 279 (2024), 824-840.

\bibitem{majdabertozzi} A.J. Majda and A.L. Bertozzi, \emph{Vorticity and incompressible flow}, Cambridge Texts in Applied Mathematics, Cambridge, 2002.

\bibitem{majda} A.J. Majda, Introduction to PDEs and Waves for the Atmosphere and Ocean, Courant Lecture Notes in Mathematics 9, AMS/CIMS, 2003.

\bibitem{ohkitani1} K. Ohkitani and J.D. Gibbon, Numerical study of singularity formation in a class of Euler and Navier-Stokes flows. Physics of Fluids 12 (2000), 3181-3194.

\bibitem{ohkitani0} K. Ohkitani and H. Okamoto, On the role of the convection term in the equations of motion of incompressible fluid, J. Phys. Soc. Japan 74 (2005), 2737-2742.

\bibitem{Okamoto1} H. Okamoto and J. Zhu, Some similarity solutions of the Navier–Stokes equations and related topics, Taiwanese J. Math. 4, No. 1 (2000), 65-103.

\bibitem{okamoto0} H. Okamoto, T. Sakajo and M. Wunsch, On a generalization of the Constantin-Lax-Majda equation, Nonlinearity 21, (2008), 2447-2461.

\bibitem{Pavlov1} M.V. Pavlov, The Gurevich-Zybin system, J. Phys. A: Math. Gen. 38 (2005), 3823-3840.

\bibitem{pedlosky} J. Pedlosky, \emph{Geophysical Fluid Dynamics}, Springer, New York 1987.

\bibitem{preston1} S.C. Preston and A. Sarria, Lagrangian aspects of the axisymmetric Euler equation, Nonlinearity 29 (2016), 1080.
  
\bibitem{Proudman1} I. Proudman and K. Johnson, Boundary-layer growth near a rear stagnation point, J. Fluid Mech. 12 (1962), 161-168.

\bibitem{Sarria0} A. Sarria and R. Saxton, Blow-up of solutions to the generalized inviscid Proudman-Johnson equation,  J. Math. Fluid Mech. 15 (2013), 493-523.

\bibitem{Sarria01} A. Sarria and R. Saxton, The role of initial curvature in solutions to the generalized inviscid Proudman-Johnson equation, Q Appl Math, 73 (1) (2015), 55-91.

\bibitem{sarria1} A. Sarria, Regularity of stagnation-point form solutions of the two-dimensional incompressible Euler equations, Differential and Integral Equations, 28 (3-4) (2015), 239-254.

\bibitem{sarria2} A. Sarria and J. Wu, Blow-up in stagnation-point form solutions of the inviscid 2d Boussinesq equations, J Differ Equations, 259 (2015), 3559-3576.

\bibitem{Saxton1} R. Saxton and F. Tiglay, Global existence of some infinite energy solutions for a perfect incompressible fluid, SIAM J. Math. Anal. 4 (2008), 1499-1515.

\bibitem{Sermange} M. Sermange and R. Temam, Some mathematical questions related to the MHD equations, Comm. Pure Appl. Math. 36 (1983), 635-664.

\bibitem{stuart} J.T. Stuart, Singularities in three-dimensional compressible Euler flows with vorticity, Theoret. Comput. Fluid Dyn. 10 (1998), 385-391.

\bibitem{swanson} C.A. Swanson, Comparison and oscillation theory of linear differential equations, Academic Press, New York, 1968.

\bibitem{Wu} J. Wu, The 2D Magnetohydrodynamic equations with partial or fractional dissipation, Lectures on the Analysis of Nonlinear Partial Differential Equations 5 (2017), 283-332.

\bibitem{Wunsch11} M. Wunsch, The generalized Proudman-Johnson equation revisited, J. Math. Fluid Mech. 13 (1) (2009), 147-154.

\bibitem{wunsch1} M. Wunsch, The generalized Hunter-Saxton system. SIAM J. Math. Anal. 42 (3) (2010), 1286-1304.

\bibitem{wunsch0} H. Wu and M. Wunsch, Global Existence for the Generalized Two-Component Hunter-Saxton System, 14 (3) (2012), 455-469.

\bibitem{Yan} W. Yan, On the explicit blowup solutions for 3D incompressible
Magnetohydrodynamics equations, \textit{arXiv:1807.07063} (Preprint). 



\end{thebibliography}
\end{document}